\documentclass[12pt,oneside,a4paper]{amsart}

\usepackage{amscd}
\usepackage{amssymb}
\usepackage{a4wide}
\usepackage{amstext}
\usepackage{amsthm}
\usepackage{mathrsfs}

\usepackage{latexsym}
\usepackage{xcolor}
\numberwithin{equation}{section}

\newcommand{\CC}{\mathbb {C}}

\newcommand{\RR}{\mathbb{R}}

 \DeclareMathOperator{\dist}{dist}

\DeclareMathOperator{\conv}{conv}
\DeclareMathOperator{\Chain}{Chain}
\DeclareMathOperator{\Ass}{Assoc}

\DeclareMathOperator{\supp}{supp} 
\DeclareMathOperator{\ospan}{\overline{Span}}

\DeclareMathOperator{\Clos}{Clos}

\renewcommand{\phi}{\varphi}
\newcommand{\pw}{\mathcal{P}W_\pi}

\newcommand{\vep}{\varepsilon}

\newcommand{\rl}{\mathbb{R}}

\newcommand{\he}{\mathcal{H}(E)}

\renewcommand{\phi}{\varphi}

\newcommand{\rea}{{\rm Re}\,}
\newcommand{\ima}{{\rm Im}\,}

\newtheorem{Thm}{Theorem}[section]
\newtheorem{theorem}[Thm]{Theorem}
\newtheorem{lemma}[Thm]{Lemma}
\newtheorem{proposition}[Thm]{Proposition}

\newtheorem{remark}[Thm]{Remark}
\newtheorem{example}[Thm]{Example}
\newtheorem{definition}[Thm]{Definition}

\begin{document}
\sloppy
\title[Synthesable differentiation-invariant subspaces]
{Synthesable differentiation-invariant subspaces}

\author{Anton Baranov, Yurii Belov}
\address{Anton Baranov,
\newline Department of Mathematics and Mechanics, St.~Petersburg State University, St.~Petersburg, Russia,
\newline National Research University Higher School of Economics, St.~Petersburg, Russia,
\newline {\tt anton.d.baranov@gmail.com}
\smallskip
\newline \phantom{x}\,\, Yurii Belov,
\newline St.~Petersburg State University, St. Petersburg, Russia,
\newline {\tt j\_b\_juri\_belov@mail.ru}}
\thanks{A. Baranov and Yu. Belov were supported by RNF grant 14-21-00035.}

\begin{abstract} We describe differentiation-invariant subspaces of $C^\infty(a,b)$ 
which admit spectral synthesis. This gives a complete answer to a question 
posed by A.~Aleman and B.~Korenblum. It turns out that this problem 
is related to a classical problem of approximation by polynomials
on the real line. We will depict an intriguing connection between 
these problems and the theory of de Branges spaces.
\end{abstract}

\maketitle


\section{Introduction and main results}

Let $C^\infty(a,b)$ be the space of all infinitely differentiable 
functions on the interval $(a,b)$ equipped with the
usual countably normed topology. A classical result of L.~Schwartz \cite{Sch}
says that any closed linear subspace of $C^\infty(\mathbb{R})$ 
which is translation invariant is generated by the exponential 
monomials $x^k e^{i\lambda x}$ it contains and, thus, the structure of 
all translation invariant subspaces is well understood. This property is 
known as {\it the spectral synthesis } for translation invariant subspaces. 
In different contexts it was studied by J. Delsart, L. Schwartz, J.-P. Kahane,
P. Malliavin, A.F. Leontiev, V.V. Napalkov, I.F. Krasichkov-Ternovskii.

In 2000s A.~Aleman and B.~Korenblum noticed that the description 
of {\it differentiation-invariant  subspaces} (or simply $D$-invariant subspaces,
where $Df = f'$) of $C^\infty(\mathbb{R})$ is
much more complicated. First results in this direction were obtained by 
A. Aleman and B. Korenblum \cite{alkor} who classified 
closed $D$-invariant subspaces $L$ of $C^\infty(a, b)$
in terms of the spectra of the restriction $D|_L$. 
Only three cases are possible: $\sigma (D|_L) = \emptyset$, 
$\sigma (D|_L) = \mathbb{C}$ and $\sigma (D|_L) = \Lambda$, where
$\Lambda$ is a discrete subset of $\mathbb{C}$.
In particular, it is shown in \cite{alkor} that any closed 
$D$-invariant subspace with void spectrum is of the form 
$$
L_I=\{f\in C^\infty(a,b): f|_I \equiv 0\},
$$
where $I$ is some relatively closed subinterval of $(a,b)$ 
(with an obvious modification in the case when $I$ reduces to one point).
Moreover, in the case $\sigma (D|_L) \ne \mathbb{C}$ there always 
exists the unique minimal interval $I$ 
such that $L_I\subset L$ (so-called {\it residual interval} for $L$).

In view of this it is natural to state the spectral 
synthesis problem for $D$-invariant subspaces as follows
(all $D$-invariant subspaces are always assumed to be closed): 
\smallskip

{\it Is it true that
any $D$-invariant subspace $L$ such that the spectrum $\sigma (D|_L)$ is discrete, 
satisfies
\begin{equation}
L=\overline{L_I+\mathcal{E}(L)},
\label{ssp}
\end{equation}
where $\mathcal{E}(L)$ is the  linear span of exponential monomials contained 
in $L$ and $I$ is the residual interval for $L$}?
\smallskip 

This problem was posed in \cite{alkor} where it was solved in the 
positive in the simplest situation when the set $\sigma (D|_L)$ is finite.  
It was further studied by A. Aleman and the authors in \cite{abb} where
it was shown that the answer in general is negative. Surprisingly, 
the answer depends essentially on the relation between 
the Beurling--Malliavin density $D^{BM}(\Lambda)$,
where $i \Lambda= \sigma(D|_L)$, and the size $|I|$ of the residual interval $I$.

\begin{theorem} 
\textup(\cite[Theorems 1.1, 1.2]{abb}\textup) 
If a $D$-invariant subspace $L$ has a compact residual interval $I$ and 
$$
2\pi D^{BM}(\Lambda) < |I|,
$$
then the spectral synthesis property \eqref{ssp} holds. 
On the other hand, there exists a $D$-invariant subspace 
$L$ with $\Lambda$ of critical density \textup(i.e., $2\pi D^{BM}(\Lambda) 
= |I|$\textup) such that \eqref{ssp} fails.

In the case when residual interval is non-compact \textup(i.e., 
$I=(a,c]$ or $I=[c,b)$\textup) or void the spectral synthesis property always holds.
\label{prevth}
\end{theorem}

The Beurling--Malliavin density $2\pi D^{BM}(\Lambda)$ 
(for the definition see \cite{BM, red})
appears naturally in this context since it is equal 
to the {\it radius of completeness} $r(\Lambda)$  of $\Lambda$,
$$
r(\Lambda):= \sup\{a: \{e^{i\lambda t}\}_{\lambda\in\Lambda} 
\text{ is complete in } L^2(-a,a)\}.
$$
The proof of Theorem  \ref{prevth} is based on the methods developed 
by M.~Mitkovski and A.~Poltoratski in \cite{mp1} (in connection with their solution
of P\'olya problem) and by the first two authors and A.~Borichev in 
\cite{bbb0}. Another approach to Theorem \ref{prevth} was suggested in \cite{abu}.

Note that the case $2\pi D^{BM}(\Lambda) > |I|$ is impossible, since then any function
in $C^\infty(I)$ can be approximated on $I$ by functions from $\mathcal{E}(L)$
together with its derivatives, and so $L = C^\infty(a,b)$.
\medskip

The present paper is devoted to the most intriguing situation 
when $2\pi D^{BM}(\Lambda) = |I|$. We are able to classify completely 
the spectra of all synthesable $D$-invariant subspaces. 
It turns out that the answer depends on 
{\it density of polynomials in some weighted spaces}.

\begin{definition}
We will say that a discrete set $\Lambda=\{\lambda_n\} \subset \mathbb{C}$ 
is synthesable if the $D$-invariant subspace $L$ of 
$C^\infty(\mathbb{R})$ which has discrete simple spectrum 
$i \Lambda$ and non-trivial residual interval $I = [-r(\Lambda),r(\Lambda)]$ is unique
\textup(and, thus, coincides with $\overline{L_I+\mathcal{E}(L)}$\textup). 
\end{definition}

To avoid uninteresting technicalities we consider only 
the case of simple spectrum. All results trivially extend
to the case of multiple spectrum and exponential monomials in place 
of exponentials.

For a nonnegative measurable function $W$ on $\mathbb{R}$ put 
$$
\mathfrak{E}_0(W)=\{F: F \text{ is entire function of zero exponential type, } 
F\in L^2(W)\}.
$$ 
The norm in the space $\mathfrak{E}_0(W)$ is inherited from $L^2(W)$. 
It is well-known that either $\mathfrak{E}_0(W)$
is dense in $L^2(W)$ or $\mathfrak{E}_0(W)$ is a (possibly zero) 
closed subspace in $L^2(W)$.

Now we are able to formulate the main result of the paper.
Note that in the case when $\mathcal{E}(L)$ is dense in $L^2(-r(\Lambda), r(\Lambda))$
or has finite codimension in this space it is not difficult 
to show that $\Lambda$ is synthesable (see Proposition \ref{old} below).
The interesting case is when the defect of $\mathcal{E}(L)$ is infinite.

\begin{theorem} 
  \label{mainth}
Let $\Lambda\subset \mathbb{C}$ be a discrete set such
that $\{e^{i\lambda t}\}_{\lambda\in\Lambda}$
has infinite codimension in $L^2(-r(\Lambda), r(\Lambda))$.
Then $\Lambda$ is synthesable if and only if 
\begin{enumerate}
\begin{item}
The product 
\begin{equation}
G(z):=\lim_{R\rightarrow\infty}\prod_{|\lambda|< R}\biggl{(}1-\frac{z}{\lambda}\biggr{)}
\label{prod}
\end{equation}
converges to an entire function $G$ of exponential type\textup;
\end{item}
\begin{item}
$\mathfrak{E}_0(|G|^2)$ contains the set  $\mathcal{P}$ of all polynomials 
and 
$$
{\rm dim}\, \Big(\mathfrak{E}_0(|G|^2) \ominus {\rm Clos}_{\mathfrak{E}_0(|G|^2)} 
\mathcal{P}\Big) \le 1
$$
\textup(the polynomials are dense in
$\mathfrak{E}_0(|G|^2)$ up to codimension one\textup).
\end{item}
\end{enumerate} 
\end{theorem}

This shows that a generic spectrum of critical density is not synthesable, 
in contrast to the non-critical density case. 
Even more surprising is that there exist two essentially 
different classes of synthesable spectra of critical density
(see Examples \ref{exa1} and \ref{exa2}) -- those for which polynomials are dense
in $\mathfrak{E}_0(|G|^2)$
and those for which polynomials are not dense, but have codimension 1.
The second case is closely related to the solution
of the long-standing spectral synthesis problem for exponential systems 
in $L^2(-\pi, \pi)$ obtained recently by the authors and A. Borichev \cite{bbb}. 

One may ask why the spectrum is not synthesable when polynomials are 
non-dense with finite defect greater than $1$. However, due to deep 
intrinsic properties of exponential systems, this 
situation can {\it never} happen.

In the case when the spectral synthesis holds there exists 
the unique $D$-invariant subspace with given residual interval 
and spectrum. Our second main result shows that
even in the case when the spectral synthesis
fails, the $D$-invariant subspaces with the same residual interval and spectrum 
have a chain structure.

\begin{theorem} Let $L_1, L_2$ be two $D$-invariant subspaces 
of $C^\infty(a,b)$ with common residual subspace $L_{1,I}=L_{2,I}$ 
\textup(or, which is the same, with common residual interval\textup). If
$$
\sigma(D|_{L_1})=\sigma(D|_{L_2}),
$$
then either $L_1\subset L_2$ or $L_2\subset L_1$.
\label{mainth2}
\end{theorem}

The crucial ingredient of the proofs is the de Branges theory 
of Hilbert spaces of entire functions \cite{br}. Though, at first glance
the studied problem is not directly related to the de Branges theory, 
there exist deep connections between them.
Note that this is the case with some other classical problems 
of harmonic analysis such that description of Fourier frames \cite{os},
Beurling--Malliavin type theorems \cite{mp}, approximation by polynomials \cite{abbb}.

Section \ref{DBTH} is devoted to our toolbox from de Branges theory. 
In Sections \ref{dinv} and \ref{necc} 
we prove, respectively, sufficiency and necessity parts
of Theorem \ref{mainth}. In Section \ref{ding}
Theorem \ref{mainth2} is proved. 
Finally, in Section \ref{EX} we give explicit 
examples of two cases when the spectral synthesis holds. 
\medskip
\\
{\bf Notations.} Given positive functions $U(x)$, $V(x)$, the notation $U(x)\lesssim V(x)$ (or, equivalently,
$V(x)\gtrsim U(x)$) means that there is a constant $C$ such that
$U(x)\leq CV(x)$ holds for all $x$ in the set in question. We write $U(x)\asymp 
V(x)$ if both $U(x)\lesssim V(x)$ and
$V(x)\lesssim U(x)$. For an entire function $F$ we denote by $\mathcal{Z}_F$
the set of its zeros (no matter what their multiplicities are). 
By $D(z, r)$ and $\overline{D}(z,r)$ we denote the disc and the closed disc
with center $z$ of radius $r$.
\bigskip


\section{De Branges theory \label{DBTH}}
In this section we briefly discuss some facts from the 
de Branges theory which are needed for the
proof of Theorem \ref{mainth}. We do not pretend to give 
an overview of the whole theory or of its essential part. 
Here we only highlight the aspects of the theory which are important for us.

\subsection{Classical theory.} 
\label{class}
All results presented here can be found in 
Sections 19--35 of the de Branges monograph \cite{br} (see, also, \cite{rom}).
An entire function $E$ is said to be 
in the Hermite--Biehler class if $|E(z)| >|E^*(z)|$,  $z\in {\mathbb{C}^+}$, where $E^* (z) = \overline {E(\overline z)}$. With any such function we associate the {\it de
Branges space} $\mathcal{H} (E) $ which consists of all entire
functions $F$ such that $F/E$ and $F^*/E$ restricted to
$\mathbb{C^+}$ belong to the Hardy space $H^2=H^2(\mathbb{C^+})$.
The inner product in $\he$ is given by
$$
(F,G)_{\he} = \int_\rl \frac{F(t)\overline{G(t)}}{|E(t)|^2} \,dt.
$$

There exists an equivalent axiomatic description 
of de Branges spaces (see \cite[Theorem 23]{br}).  
A nontrivial Hilbert space $\mathcal{H}$ whose elements are entire functions 
is a de Branges space $\mathcal{H} (E) $ (for some $E$)
if and only if it satisfies the following axioms:
\begin{itemize}
\item[(A1)] For every nonreal number $w$, the evaluation functional $F\mapsto F(w)$ is continuous;
\item[(A2)] Whenever $F$ is in the space and has a nonreal zero $w$, the function $F(z)\frac{z-\bar{w}}{z-w}$ is in the space and has the same norm as $F$;
\item[(A3)] The function $F^*$ belongs to the space whenever $F$ belongs to the space 
and has the same norm as $F$.
\end{itemize}

In what follows we require additional assumption that for any 
$w\in\mathbb{C}$ there exists  $F\in\mathcal{H}$ such that 
$F(w)\neq0$. This corresponds to the situation when $E$ has 
no real zeros. 
 
The prime examples of de Branges spaces are Paley--Wiener spaces 
$\mathcal{P}W_a$, $a>0$. Recall that $\mathcal{P}W_a$ is the space
of all entire functions of exponential type at most $a$
whose restriction to $\mathbb{R}$ belong to $L^2(\mathbb{R})$
(with the norm inherited from $L^2(\mathbb{R})$). Equivalently, 
$\mathcal{P}W_a$ is the Fourier image of $L^2(-a,a)$.

It is clear that the spaces $\mathcal{P}W_a$ 
are ordered by inclusion (with equality of norms).
One of the main results of de Branges theory 
says that this property is true for general de Branges spaces. 
For a de Branges space $\he$ consider the set of all its de Branges subspaces
$$
\Chain(\he):=\big\{\mathcal{H}: \mathcal{H} -\text{de Branges space}, 
\mathcal{H}\subset\he, \|\cdot\|_{\mathcal{H}}=\|\cdot\|_{\he}\big\}.
$$          
Then subspaces in $\Chain(\he)$ are ordered by inclusion.

\begin{theorem} If $\mathcal{H}_1,\mathcal{H}_2\in\Chain(\he)$, 
then either $\mathcal{H}_1\subset\mathcal{H}_2$
or $\mathcal{H}_2\subset\mathcal{H}_1$.
\label{ord}
\end{theorem}

Note that if a de Branges space $\he$ contains the linear space $\mathcal{P}$
of all polynomials and $\Clos_{\he} \mathcal{P} = \he$, then  
$\Chain(\he) = \{\mathcal{P}_n:\ n\in\mathbb{N}_0\} \cup \{\he\}$, 
where $\mathcal{P}_n$ is the set of polynomials of degree at most $n$.
                                                  
\begin{theorem}
Let $\mathcal{H}$ be a de Branges subspace of $\he$. If codimension of $\mathcal{H}$ 
in $\he$ is equal to  $n\in\mathbb{N}$, then $\mathcal{H}$ 
is the closure of the domain of multiplication by $z^n$ in $\he$,
$$
\mathcal{H}={\rm Clos}_{\he} \big\{f\in\he: z^nf\in\he\big\}.
$$
Moreover, in this case for any $m<n$ there exists a de Branges 
subspace of codimension $m$ given by ${\rm Clos}_{\he}
\{f\in\he: z^m f\in\he\}$.

In particular, if $\mathcal{H}$ is a proper de Branges subspace 
of $\he$ and $\he$ has no de Branges subspace of codimension $1$,
then $\dim(\he\ominus\mathcal{H})=\infty$. 
\label{finth}
\end{theorem}

\begin{remark}
\label{afun}
{\rm Let $\mathcal{H}_0$ be a de Branges subspace
of $\he$ of codimension $1$. Then, by \cite[Problem 87]{br}, one can choose
$E_0=A_0 -iB_0$ such that $\mathcal{H}_0 = \mathcal{H}(E_0)$ and $A=A_0$, 
$B = zA_0 +B_0$. Note that in this case $A_0 \in \he$. } 
\end{remark}

One of the important notions in the de Branges theory is an {\it associated function}.

\begin{definition}
We will say that an entire function $G$ is an associated function 
for $\he$ and write $G\in\Ass(\he)$ if, for any $w\in\mathbb{C}$ 
and $F\in\he$,
$$
\frac{F(z)G(w)-G(z)F(w)}{z-w}\in\he.
$$
\end{definition}

Of course, we have $\he\subset\Ass(\he)$ and also $E\in \Ass(\he)$. 

In what follows we will often use the representation $E=A-iB$, where
$A=\frac{E+E^*}{2}$, $B=\frac{E^*-E}{2i}$ are entire functions which are 
real on $\mathbb{R}$. 
We will say that the corresponding function $A$ is {\it $A$-function 
of the de Branges space $\he$.} 
Note that $\mathcal{H}(e^{i\alpha} E) = \he$ for any $\alpha \in [0, 2\pi)$ 
and the $A$-function for $e^{i\alpha} E$ is given by 
$\cos \alpha A - \sin \alpha B$. Moreover, recall 
that the zero sets of the functions $\cos \alpha A - \sin \alpha B$ generate 
orthogonal bases of reproducing kernels 
for all $\alpha \in [0, 2\pi)$ except at most one. 
Such exceptional $\alpha$ exists if and only if  
$\he$ contains a subspace of codimension 1.
In what follows we always assume without loss of generality
that the set $\mathcal{Z}_A$ generates 
an orthogonal basis of reproducing kernels which is, up to normalization,
of the form $\big\{\frac{A}{z-\mu}, \ \mu \in \mathcal{Z}_A\big\}$.

In what follows we will need the following technical lemma.

\begin{lemma}
\label{af}
Let $\mathcal{H}$ be a de Branges space 
and let $\tilde{\mathcal{H}}$ be its de Branges subspace of infinite codimension.
Let $A$ and $\tilde{A}$ be the $A$-functions of  
$\mathcal{H}$ and $\tilde{\mathcal{H}}$ respectively. 
If the function $A$ \textup(and, hence, any element in $\mathcal{H}$\textup)
is of order at most, one then $|\tilde{A}(iy)\slash A(iy)|=o(|y|^{-N})$, $|y|\to \infty$,
for any $N>0$. 
\end{lemma}

\begin{proof} 
Note that both $A$ and $\tilde A$ are canonical 
products of genus 0 or 1 with real zeros.
Therefore, if $t_n$ are zeros of $A$, 
then $|A(iy)|^2 = \prod_n (1+y^2/t_n^2)$.

It follows from \cite[Problem 93]{br} that between each two zeros of $\tilde A$
there is at least one zero of $A$. 
Hence, if infinitely many intervals between two consequitive zeros of $\tilde A$
contain more than one zero, then it is easy to see that
$|\tilde{A}(iy)\slash A(iy)|=o(|y|^{-N})$, $|y|\to \infty$, for any $N>0$.

If, otherwise, all except a finite number of 
intervals between two consequitive zeros of $\tilde A$
contain exactly one zero, then there exist $N$ such that 
$|\tilde{A}(iy)|\gtrsim |y|^{-N}|A(iy)|$, $|y|\to \infty$. 
In this case \cite[Theorem 26]{br} implies that for any $F\in\mathcal{H}$ 
and any polynomial $P$ of degree $N+2$ which divides $F$, we have 
$F/P \in \tilde{\mathcal{H}}$. Thus, $\tilde{\mathcal{H}}$ contains 
the domain of multiplication by $z^{N+2}$ and so 
$\tilde{\mathcal{H}}$ has the codimension at most $N+2$ in $\mathcal{H}$,
a contradiction.
\end{proof}


\subsection{Recent progress.} 
\label{recent}
In this subsection several facts 
recently discovered by the authors are collected. The next result 
was used in \cite{abbb}.

\begin{theorem}
Let $G$ be an entire function with simple zeros, $G^*\slash G$ be a ratio of two Blaschke products and $G\in\Ass(\he)$. Put
$$
\mathcal{H}=\ospan\biggl{\{}\frac{G(z)}{z-w}, \ w\in\mathcal{Z}_G\biggr{\}}.
$$
Then $\mathcal{H}$ is a de Branges subspace of $\he$.
\label{spaceG}
\end{theorem}

\begin{proof} 
The proof follows easily from the axiomatic description of de Branges spaces.
Indeed, it is clear that $\ospan \big\{\frac{G(z)}{z-w}, \ w\in\mathcal{Z}_G \}$
is closed under division by Blaschke factors, and so $\mathcal{H}$
is closed under division by Blaschke products. 
Since $G^*\slash G$ is a ratio of two Blaschke products, we conclude that
the function 
$$
\frac{G^*(z)}{z-w} = \frac{G(z)}{z-w} \cdot\frac{z-w}{z-\bar w}\cdot
\frac{G(z)}{G^*(z)}
$$
is in $\mathcal{H}$ for any $w\in\mathcal{Z}_G$, and so $\mathcal{H}$
is closed under the transform $F\mapsto F^*$.
\end{proof}

Let $\{e^{i\lambda t}\}_{\lambda\in\Lambda}$ be an 
incomplete system in $L^2(-\pi,\pi)$. 
In this case $\Lambda$ is a subset of the zero set of some nontrivial function 
in $\pw$ and so $\Lambda$ satisfies the Blaschke conditions 
in $\mathbb{C}^+$ and in $\mathbb{C}^-$.
Fix some canonical product $G_\Lambda$
with simple zeros at $\Lambda$ such that $G_\Lambda^*/G_\Lambda$ is a ratio 
of two Blaschke products. Put 
\begin{equation}
\label{fact}
\mathcal{H}_{\Lambda,\pi}:=\{F: F \text{ is entire and } G_\Lambda F\in\pw\}.
\end{equation}
Define the norm in $\mathcal{H}_{\Lambda,\pi}$ by the formula 
$\|F\|_{\mathcal{H}_{\Lambda,\pi}} = \|G_\Lambda F\|_{\pw}$.
The following result was proved in \cite[p. 217]{B1} 
using axiomatic description of de Branges spaces.

\begin{theorem}
The space $\mathcal{H}_{\Lambda,\pi}$ is a de Branges space.
\end{theorem}

As we will see below, the space $\mathcal{H}_{\Lambda,\pi}$ is closely connected 
with $D$-invariant subspaces (see Section \ref{dinv}).
The space $\mathcal{H}_{\Lambda,\pi}$ is ``generated by'' the 
weight $|G_\Lambda|^2$. This 
produces some restrictions on the structure of $\Chain(\mathcal{H}_{\Lambda,\pi})$.

\begin{theorem}
If $\mathcal{H}$ is an infinite-dimensional de Branges subspace 
of $\mathcal{H}_{\Lambda,\pi}$, then there exists a subspace 
of $\mathcal{H}$ of infinite codimension in $\mathcal{H}$. Moreover, 
there exists $f\in\mathcal{H}$ such that $z^nf(z)\in\mathcal{H}$ 
for any $n\in\mathbb{N}$.
\label{notleft}
\end{theorem}

\begin{proof}
First, let us fix a function $F\in\mathcal{H}$ with infinite number of zeros.
Let $\{s_n\}^\infty_{n=1}$ be such that $\{s_n\}\subset\mathcal{Z}_F$, 
$|s_{n+1}|>10|s_n|$ and either $|\ima s_n -1|>\frac{1}{4}$, $n\in\mathbb{N}$, or 
$|\ima s_n|>\frac{1}{4}$, $n\in\mathbb{N}$.
Put
$$
 S(z)=\prod_n\biggl{(}1-\frac{z}{s_n}\biggr{)}.
$$
The function $S$ satisfies either $|S(x)|\gtrsim 1$, $x\in\mathbb{R}$, 
or $|S(x+i)|\gtrsim1$, $x\in\mathbb{R}$. 
Hence, $G_\Lambda F S^{-1}\in\pw$. 
Using \cite[Theorem 26]{br} we conclude that $FS^{-1}\in\mathcal{H}$.
Moreover, by analogous arguments we can prove that 
$z^n F(z)S^{-1}(z)\in\mathcal{H}$ for any $n\in\mathbb{N}$.  

Now assume that any subspace in $\Chain(\mathcal{H})$ 
has finite codimension. Then either the chain is finite or 
$\cap_{\mathcal{H}_1 \in \Chain(\mathcal{H})} \mathcal{H}_1= \emptyset$.
On the other hand, $FS^{-1}$ is in the domain of multiplication
by $z^n$ for any $n$ and it follows from Theorem \ref{finth} 
that $FS^{-1}\in\mathcal{H}_1$ for any de Branges subspace $\mathcal{H}_1$ 
of finite codimension in $\mathcal{H}$. We arrive to a contradiction.
\end{proof}

In what follows we denote by $k_\lambda = k^{\pw}_\lambda$ the reproducing kernel 
of $\pw$, i.e., the cardinal sine function:
$k_\lambda(z) = \frac{\sin \pi(z-\overline{\lambda})}{\pi(z-\overline{\lambda})}$.

\begin{theorem}
\label{twoint}
If $\mathcal{H}$ is a de Branges subspace of $\mathcal{H}_{\Lambda,\pi}$, then
$$
\dim(\mathcal{H}_{\Lambda,\pi}\ominus\mathcal{H})\neq2.
$$
\end{theorem}

\begin{proof}
Assume the contrary. It is well known that existence of a de Branges 
subspace of codimension $1$ is equivalent to the condition
$\mu(\mathbb{R})<\infty$, where $\mu$ is the measure 
from the Herglotz representation of $(E+e^{i\alpha}E^*)/(E-e^{i\alpha}E^*)$
for some $\alpha \in [0, 2\pi)$ (a so-called Clark measure). It is easy 
to prove that a de Branges space contains 
a de Branges subspace of codimension $2$ if and only if  
$\int_\mathbb{R}|x|d\mu(x)<\infty$ (as such subspace 
one should take $\Clos_{\he} \{F: z^2F\in \he\}$).

Moreover, in this situation in $\mathcal{H}_{\Lambda,\pi}$ 
there exists a complete and minimal system of
reproducing kernels $\{K_{t}\}_{t\in T}$ 
of the space $\mathcal{H}_{\Lambda,\pi}$ such 
that its biorthogonal system $\frac{G_T(z)}{G'_T(t)(z-t)}$ 
has codimension at least 2 (see \cite[Proposition 9.1]{bbb1}). 
By Theorem \ref{spaceG}, $\ospan\big\{\frac{G_T(z)}{z-t}\big\}$ 
is a de Branges subspace of $\mathcal{H}_{\Lambda,\pi}$ 
of codimension 2. Then it is clear from the ordering theorem that
$\ospan\big\{\frac{G_T(z)}{z-t}\big\} = \mathcal{H}$. 

Without loss of generality we can assume that $\Lambda\cap T=\emptyset$. Put
$$
\mathcal{MS}:=\{k^{\pw}_\lambda\}_{\lambda\in\Lambda}
\cup\biggl{\{}\frac{G_\Lambda G_T}{z-t}\biggr{\}}_{t\in T}.
$$
This is a so-called {\it mixed system} in the 
Paley--Wiener space. First of all we note that 
the system $\mathcal{S}:=\{k^{\pw}_\lambda\}_{\lambda\in\Lambda\cup T}$ 
is complete and minimal in $\pw$. Indeed, if $H\in\pw\setminus\{0\}$ 
is orthogonal to $\mathcal{S}$,  then $H = G_\Lambda G_T S$ for some entire $S$. 
Hence,  the function $G_T S$ belongs to $\mathcal{H}_{\Lambda,\pi}$ and 
is orthogonal to $\{K_t\}_{t\in T}$, a contradiction. Minimality 
of $\mathcal{S}$ also follows immediately from minimality 
of $\{K_t\}_{t\in T}$.

Assume that there exist two linearly independent 
elements $H_1,H_2$ in $\mathcal{H}_{\Lambda,\pi}$
such that $H_1,H_2 \perp \mathcal{H}$ or, equivalently  
$$
H_1,\  H_2 \perp \biggl{\{}\frac{G_T(z)}{z-t}\biggr{\}}_{t\in T}.
$$
Put $F_1=G_\Lambda H_2$, $F_2=G_\Lambda H_2$. Then we have $F_{1,2}\perp\mathcal{MS}$
in $\pw$. This contradicts Theorem 1.1  in \cite{bbb}
which says that if $\mathcal{S}$ is a complete and minimal system 
of reproducing kernels in $\pw$, then the orthogonal complement 
to any mixed system of the form $\mathcal{MS}$ is at most one-dimensional. 
\end{proof}
\bigskip


\section{Proof of Theorem \ref{mainth}: Sufficiency}
\label{dinv}

\subsection{Preliminary steps.} 
\label{prelim}
Recall that continuous linear functionals on $C^\infty(a,b)$  are compactly supported
distributions of finite order on the interval $(a, b)$, 
i.e., distributions of the form 
$\phi=\sum_{k=0}^n c_k h_k^{(k)}$, where $c_k\in\mathbb{C}$, 
$h_k\in L^2(a,b)$ have compact supports 
$[c_k,d_k] \subset (a,b)$ and differentiation
is understood in the sense of distributions. We denote this class 
of distributions by $S(a,b)$.
Then the action of $\phi \in S(a,b)$ on $f\in C^\infty(a,b)$ is given by
$$
\phi(f) = \sum_{k=0}^n c_k (-1)^k \int_{c_k}^{d_k} h_k(t) f^{(k)}(t)\, dt.
$$

As usual we can define the Fourier transform of $\varphi\in S(a,b)$ by the formula
$$
\hat{\varphi}(z)=\varphi(e^{izt}).
$$
Since $\phi$ has finite order, the function $\hat{\varphi}$ is an entire function
of finite exponential type with at most polynomial growth on the real line.
Note that in the case when $\phi$ is a usual $L^2$ function (not a distribution) 
supported by the interval $[c, d] \subset (a,b)$ 
we will understand $\phi(f)$ as the usual inner product (with $\bar f$):
$$
\phi(f) = \int_c^d \phi(t) f(t)\, dt = 
\int_\mathbb{R}  \widehat{\phi}(x) \widehat{\bar f|_{[c,d]}}(x)\, dx,
$$
where $\hat g(x) = (2\pi)^{-1/2} \int_\mathbb{R} g(t)e^{itx} dt$ 
is the usual Fourier transform (we write $e^{itx}$ in place of $e^{-itx}$ 
to agree with the Fourier transform on $S(a,b)$).

In what follows we will assume that $L$ is a $D$-invariant subspace with
the spectrum $i \Lambda$ such that $D^{BM}(\Lambda) = 1$ and with the
residual interval $I=[-\pi,\pi]$. 

We will sometimes 
use the following simple observation: if $L$ is a $D$-invariant 
subspace with spectrum $i \Lambda$, then $e^{it\gamma}L$ is also $D$-invariant 
with the spectrum $i(\Lambda +\gamma)$. Clearly, $L$ and $e^{it\gamma}L$ 
admit spectral synthesis or not simultaneously. In particular, we can always
assume in what follows that $\Lambda\cap \mathbb{Z} = \emptyset$ 
(and even some weak separation similar to formula \eqref{bbb} below).

Let us fix some canonical product $G_\Lambda$ with zero set 
$\Lambda$ such that $G_\Lambda(z)\slash G^*_\Lambda(z)$ is a ratio
of two Blaschke products (since $\Lambda$ has a finite completeness radius, 
it satisfies the Blaschke conditions in $\mathbb{C^+}$ and in $\mathbb{C}^-$).
We will be mainly interested in the situation when the system 
$\{e^{i\lambda t}\}_{\lambda\in\Lambda}$ has infinite codimension in $L^2(-\pi,\pi)$.
In this case the space $\mathcal{H}_{\Lambda,\pi}$ (defined by \eqref{fact})
is nontrivial and, moreover, $\dim(\mathcal{H}_{\Lambda,\pi})=\infty$.
The following lemma establishes the link between Theorem \ref{mainth}
and the constructions of Subsection \ref{recent}. 

\begin{lemma} 
\label{link}
If $\Lambda$ and $G_\Lambda = G$ satisfy the conditions {\rm (i)} and {\rm (ii)} 
of Theorem \ref{mainth}, then 
$\mathfrak{E}_0(|G_\Lambda|^2) = \mathcal{H}_{\Lambda,\pi}$
as Hilbert spaces with equality of norms. 
\end{lemma}

\begin{proof} 
Since $\mathcal{P} \subset \mathfrak{E}_0(|G_\Lambda|^2)$, we conclude that
$G_\Lambda \in L^2(\mathbb{R})$. Any function of order 1 representable
as the principal value product \eqref{prod} is of finite 
exponential type, and so $G_\Lambda$ is in some Paley--Wiener space. 
In particular, $G_\Lambda$ belongs to the Cartwright class (for the properties 
of this class of entire functions see, e.g., \cite{hj, lev}). 
Therefore, the Beurling--Malliavin 
density for $\Lambda$ coincides with the usual density
$\lim_{r\to\infty} (2r)^{-1} \# (\Lambda \cap D(0, r))$
and coincides with the type of $G_\Lambda$ divided by $\pi$. We 
conclude that the type of $G_\Lambda$ equals $\pi$ and 
its indicator diagram is the interval $[-\pi i, \pi i]$. 
                                                    
Let $F \in \mathfrak{E}_0(|G_\Lambda|^2)$.  
Since $G_\Lambda F \in L^2(\mathbb{R})$ and $F$ is of zero exponential type,
$G_\Lambda F \in \pw$. Thus, $\mathfrak{E}_0(|G_\Lambda|^2) 
\subset  \mathcal{H}_{\Lambda,\pi}$. 

Conversely, if $G_\Lambda F \in \pw$, then $F$ is of zero exponential type
since $G_\Lambda$ has the maximal (for $\pw$) indicator diagram. Hence, 
$F\in \mathfrak{E}_0(|G_\Lambda|^2)$. 
\end{proof}

Thus, in what follows we can replace the condition (ii) of Theorem \ref{mainth}
by the equivalent condition: {\it polynomials belong to 
$\mathcal{H}_{\Lambda,\pi}$ and are dense there up to codimension 1.}

Now put
$$
L_0=\overline{L_I+\mathcal{E}(L)},
$$
We consider the annihilators $L^{\perp}$ and $L^\perp_0$ in the dual space $S(a,b)$. 
Since $L_I\subset L_0$, it is clear that any 
$\phi \in L^\perp_0$ should be supported by the interval
$[-\pi, \pi]$. Also, 
$\hat \phi(\lambda) = \phi(e^{i\lambda t}) = 0$ 
whenever $e^{i\lambda t}$. Hence, 
$$
\widehat{L^\perp_0}=\{F: F\in\mathcal{E}_\pi,\quad F\bigl{|}_\Lambda=0\},
$$
where 
$$
\mathcal{E}_{\pi}=\biggl{\{}F: F = \sum_{k=0}^n z^k(F_k(z)+c_k),
\quad F_k\in\pw, \ c_k \in \CC \biggr{\}}.
$$
By the Hahn--Banach Theorem, the equality $L=L_0$ is 
equivalent to the equality $L^\perp = L^\perp_0$ 
or to $\widehat{L^\perp} = \widehat{L^\perp_0}$,
where $\widehat{L^\perp} = \{\hat \phi: \phi\in L^\perp\}$.

Since $L$ is $D$-invariant, it is immediate from the duality that 
for any $\phi \in L^\perp$ and any $n \in\mathbb{N}$, 
we have $z^n \hat\phi(z) \in \widehat{L^\perp}$.
It was shown in \cite[Proposition 3.1]{alkor} that 
the class $\widehat{L^\perp}$ is also closed under dividing out zeros
which are not in the spectrum:
$$
\text{if} \quad \phi \in L^\perp \quad\text{and} \quad \hat\phi(w) = 0, \ w\notin\Lambda, 
\quad \text{then} \quad \frac{\hat\phi(z)}{z-w} \in \widehat{L^\perp}.
$$
This observation plays a crucial role in what follows.
Clearly, we can also divide by $z-w$ if $w\in\Lambda$, but 
multiplicity of the zero $w$ is greater than its multiplicity in the spectrum.

Statement 2 of the following proposition 
was proved in \cite{abb}. We include 
the proof for the reader's convenience.

\begin{proposition}
\label{old}
1. Assume that $\{e^{i\lambda t}\}_{\lambda\in\Lambda}$ is complete 
in $L^2(-\pi,\pi)$ or has finite codimension there. 
Then $L = L_0$.
\smallskip

2. Assume that $\{e^{i\lambda t}\}_{\lambda\in\Lambda}$ has infinite codimension in 
$L^2(-\pi,\pi)$ and $L\ne L_0$. 
Also, assume that the set $(\rea \Lambda)\,{\rm mod}\,\mathbb{Z}$ is infinite, i.e., 
$\rea \Lambda$ is not contained in a finite union of progressions 
of the form $\mathbb{Z} +\gamma$, $\gamma \in \mathbb{R}$.
Then there exist entire functions $E$ and $T$
and a function $F\in\pw$ such that $E$ has infinitely many zeros, 
$\mathcal{Z}_E\cap\Lambda=\emptyset$, $G_\Lambda E
\in\pw \cap \widehat{L^\perp}$, $G_\Lambda T\in\pw$, and 
\begin{equation}
\begin{cases}
\int_\mathbb{R}\frac{G_\Lambda(x)E(x)}{x-w}\overline{F(x)}dx  = 0,\quad w\in\mathcal{Z}_E,\\
\int_\mathbb{R}G_\Lambda(x)T(x)\overline{F(x)} dx  \neq  0.
\label{inteq}
\end{cases}
\end{equation}
\end{proposition}

\begin{proof}
Assume that $L\ne L_0$ and let $\varphi\in L^\perp\setminus\{0\}$.
Since $\phi$ annihilates exponentials $e^{i\lambda t}$, we have 
$\hat{\varphi}(\lambda) = 0$, $\lambda\in\Lambda$, whence for some entire function $E$, 
$$
\hat{\varphi}=G_\Lambda E \qquad \text{and} \qquad G_\Lambda E\in \mathcal{E}_\pi.
$$
Since the class $\widehat{L^\perp}$ is closed under dividing out zeros, 
we have 
\begin{equation}
\label{ghj}
\frac{G_\Lambda(z) E(z)}{z-w}\in \widehat{L^\perp},\quad w\in\mathcal{Z}_E.
\end{equation}
We will assume without loss of generality that  
$\mathcal{Z}_E\cap\Lambda=\emptyset$ and $E$ has no multiple zeros.
\medskip

{\bf Step 1: Proof of statement 1.}
First we consider the case 
when any function $E$ with the above properties is a polynomial. 
This is the case, in particular, when 
$\{e^{i\lambda t}\}_{\lambda\in\Lambda}$ is complete 
in $L^2(-\pi,\pi)$ or has finite codimension there. Indeed, 
$G_\Lambda E\in \mathcal{E}_\pi$. Therefore, if $E$ has infinitely many zeros, 
then we have infinitely many linear independent elements of 
$\pw$ of the form $G_\Lambda E/P$, where $P$
is a polynomial with $\mathcal{Z}_P \subset \mathcal{Z}_E$
of sufficiently large degree. This contradicts the fact 
that the system of reproducing kernels $\{k_\lambda\}_{\lambda\in \Lambda}$
(the Fourier image of $\{e^{i\lambda t}\}_{\lambda\in\Lambda}$)
is complete in $\pw$ or has a finite codimension there. 
Also, since $G_\Lambda$ has completeness radius 1, it 
is impossible that $E(z) = e^{az}$ for some $a\in\mathbb{C}\setminus\{0\}$.

If $E$ is a polynomial, it follows from \eqref{ghj} that $G_\Lambda \in
\widehat{L^\perp}$, whence $z^k G_\Lambda \in \widehat{L^\perp}$ 
for any $k\in\mathbb{N}$. 

Now let $\psi \in \widehat{L_0^\perp}$. Then $\hat\psi = G_\Lambda T 
\in \mathcal{E}_\pi$ for some entire function $T$ and arguing as above 
we conclude that $T$ is a polynomial. Thus, 
$\widehat{L^\perp} = \widehat{L^\perp_0}$ and so $L=L_0$.
\medskip

{\bf Step 2: Proof of statement 2.} If $L \ne L_0$, then by Step 1
there exists a function $E$ as above which has infinitely many zeros.
In view of \eqref{ghj} we can start with the function 
$\frac{E(z)}{(z-w_1)...(z-w_n)}$ in place of $E$ 
where $n$ is sufficiently large. Thus, 
we can assume that $G_\Lambda E\in\pw$. 
Denote by $\phi_w$, $w\in \mathcal{Z}_E$, the functional in 
$L^\perp$ such that $\frac{G_\Lambda(z) E(z)}{z-w} = \hat \phi_w (z)$.

Since $L \ne L_0$, there exist $\psi\in L_0^\perp$ 
and $f_1\in C^\infty([-\pi,\pi])$ such that 
$$
\phi(\bar f_1) = 0, \ \  \phi \in \widehat{L^\perp}, 
\qquad \text{but} \qquad \psi(\bar f_1) \ne 0.
$$ 
In particular, $\phi_w(\bar f_1) = 0$, $w\in\mathcal{Z}_E$.
Since $G_\Lambda E \in \pw$, this condition can be rewritten as
$$
\int_\mathbb{R}\frac{G_\Lambda(x)E(x)}{x-w}\overline{F_1(x)}dx  = 0,
\qquad w\in\mathcal{Z}_E,
$$
where $F_1= \hat f_1$. However, for $\psi\in L_0^\perp$ 
we know only that $\hat \psi \in \mathcal{E}_\pi$ 
and so $\hat \psi = P G_\Lambda T$ where  
$G_\Lambda T = \hat h \in \pw$ and $P$ is a polynomial 
of some degree $N$. If we write 
$P(z) = \sum_{k=0}^N c_k z^k$, then 
$$
\psi = \sum_{k=0}^N c_k i^k h^{(k)}.
$$

The functionals $\phi_w$ and $\psi$ annihilate $e^{i\lambda t}$, 
$\lambda \in \Lambda$. Therefore, we can 
replace $f_1$ by a function $f_2$ such that $\bar f_2(t) = 
\bar f_1(t)-\sum_{j\in J} c_j e^{i \lambda_j t}$ 
for any finite set $\{\lambda_j\}_{j\in J} \subset \Lambda$. 
Since the set $(\rea \Lambda)\,{\rm mod}\,\mathbb{Z}$ is infinite, 
we can choose
$\lambda_j$ so that $f_2^{\ell}(\pm \pi) =0$, $0\le \ell \le N$. 
Therefore, for $F_2 = \hat f_2$ we have $(-iz)^j F_2
=\widehat{f_2^{(j)}} \in \pw$, $0\le j\le N$,
and so $P^* F_2 \in \pw$. Hence, 
$$
\begin{aligned}
\psi(\bar f_1) & = \psi(\bar f_2) = 
\sum_{k=1}^N (-i)^k c_k \int_{-\pi}^\pi h(t)\overline{f_2^{(k)}(t)} dt \\
& =
\sum_{k=1}^N (-i)^k c_k 
\int_\mathbb{R} G_\Lambda(x) T(x) \overline{(-ix)^k F_2(x)} dx = 
\int_\mathbb{R} G_\Lambda(x) T(x) \overline{P^*(x)} \overline{F_2(x)} dx
\end{aligned}
$$
and we conclude that $\int_\mathbb{R} 
G_\Lambda T \overline{P^*} \overline{F_2} \ne 0$.

On the other hand, dividing $E$ by sufficiently many factors of the form $z-w_l$,
$w_l\in \mathcal{Z}_E$, we can assume that 
$z^k \frac{G_\Lambda(z) E(z)}{z-w} \in \pw$, $k=0, 1, \dots, N$.
Since $z^k \frac{G_\Lambda(z) E(z)}{z-w} \in \widehat{L^\perp}$, 
we can write $P(z)\frac{G_\Lambda(z) E(z)}{z-w} = \hat \eta_w (z)$, 
where $\eta_w \in L^\perp$. Then we have 
$$
\begin{aligned}
\int_\mathbb{R} \frac{G_\Lambda(x) E(x)}
{x-w}  \overline{P^*(x)} \overline{F_2(x)} dx
& = \int_\mathbb{R} P(x) \frac{G_\Lambda(x) E(x)}
{x-w}  \overline{F_2(x)} dx  \\
& = \int_\mathbb{R} \hat\eta_w(x) \overline{F_2(x)} dx
= \eta_w (\bar f_2) = \eta_w (\bar f_1) =0, \qquad w\in\mathcal{Z}_E.
\end{aligned}
$$
Thus, the functions $E$, $T$ and $F = P^* F_2$ satisfy \eqref{inteq}.
\end{proof}
\medskip
 
Note that we also have $G_\Lambda(z) E(z)z^n\in \widehat{L^{\perp}}$.  So,
for $E$ and $F$ from Proposition \ref{old}, we have
\begin{equation}
\begin{cases}
 \int_\mathbb{R}\frac{G_\Lambda(x)E(x)x^n}{x-w}\overline{F(x)}dx=0,
 \quad w\in\mathcal{Z}_E,\\
 \int_\mathbb{R}G_\Lambda(x)T(x)\overline{F(x)} dx \neq 0,
\label{inteq2}
\end{cases}
\end{equation} 
for any $n\in\mathbb{N}_0$ such that $G_\Lambda(z)E(z)z^n\in\pw$.
\medskip

There is one technical problem in the above proof in the case
when the set $(\rea \Lambda)\,{\rm mod}\,\mathbb{Z}$ is finite
(i.e., $\rea \Lambda$ is contained in a finite union of progressions of the form 
$\mathbb{Z} +\gamma$). Indeed, 
in this case the exponentials $e^{i\lambda t}$ may take only finite number of values at 
$\pm \pi$ and we cannot guarantee the existence of 
their linear combination $\sum_{j\in J} c_j e^{i \lambda_j t}$
which coincides at $\pm\pi$ with $f_1$ up to derivative of order $N$. 
This difficulty can be overcome by a simple perturbation argument. 

In what follows we say that a sequence 
$\{z_n\}_{n\in\mathbb{N}}$ is {\it lacunary} if 
$\liminf_{n\to\infty} |z_{n+1}|/|z_n| >1$. 
By a {\it lacunary canonical product} we will mean a zero genus canonical product
with a lacunary zero set.

\begin{proposition}
\label{tech}
Let the set $(\rea \Lambda)\,{\rm mod}\,\mathbb{Z}$ be finite
Assume that $\{e^{i\lambda t}\}_{\lambda\in\Lambda}$ has infinite codimension in 
$L^2(-\pi,\pi)$ and $L\ne L_0$. 
Then there exists a sequence $\tilde \Lambda$ such that 
for the corresponding function $G_{\tilde \Lambda}$ we have 
\begin{equation}
\label{skr}
|G_\Lambda(z)| \asymp |G_{\tilde \Lambda}(z)|, \qquad {\rm dist}\,(z, \mathbb{R})\ge 1/10,
\end{equation}
and entire functions $E$, $T$ and $F\in\pw$ such that $E$ has infinitely many zeros, 
$\mathcal{Z}_E\cap\Lambda=\emptyset$, $G_{\tilde \Lambda} E, G_{\tilde \Lambda} 
T\in\pw$, and 
\begin{equation}
\begin{cases}
\int_\mathbb{R}\frac{G_{\tilde \Lambda}(x)E(x)}{x-w}\overline{F(x)}dx  = 0,\quad w\in\mathcal{Z}_E,\\
\int_\mathbb{R}G_{\tilde \Lambda}(x)T(x)\overline{F(x)} dx  \neq  0.
\label{inteq-1}
\end{cases}
\end{equation}
\end{proposition}

\begin{proof} 
Arguing as in the beginning of the proof of Proposition \ref{old}, 
we can find $E$ and $F_1\in\pw$, $F_1 = \hat f_1$, such that
$$
\phi_w(\bar f_1) = 
\int_\mathbb{R}\frac{G_\Lambda(x)E(x)}{x-w}\overline{F_1(x)}dx = 0,
\qquad w\in\mathcal{Z}_E,
\qquad \text{but} \qquad \psi(\bar f_1) \ne 0.
$$ 
We need to replace $\Lambda$ by a perturbed sequence $\tilde\Lambda$ 
preserving these properties. 

Without loss of generality, ${\rm \dist}\,(\Lambda, \mathbb{Z}) >0$ 
and $\mathcal{Z}_E \cap\mathbb{Z} = \emptyset$. Let us expand $F_1$ into the series with respect 
to the basis of cardinal sine functions, 
$F_1 = \sum_{n\in\mathbb{Z}} \bar a_n k_n$. Then the equation 
$$
\int_\mathbb{R}\frac{G_\Lambda(x)E(x)}{x-w}\overline{F_1(x)}dx = 0
$$
is equivalent to
\begin{equation}
\label{trex}
\sum_{n\in\mathbb{Z}}\frac{G_\Lambda(n)E(n)a_n}{n-w} = 0.
\end{equation}
Now put
$$
G_{\tilde \Lambda} = G_\Lambda + \delta  \frac{Q_1}{Q_2} G_\Lambda,
$$
where $\delta>0$ and $Q_1, Q_2$ are two lacunary canonical products 
such that $\mathcal{Z}_{Q_2} \subset \Lambda$ 
and $|Q_1(n)/Q_2(n)| = o(|n|^{-M})$, $n\in\mathbb{Z}$, $n\to \infty$,
for any $M>0$. E.g., one can take 
$\mathcal{Z}_{Q_2} = \{\lambda_{n_k}\}_{k\in \mathbb{N}}$ to be 
some lacunary subsequence of $\Lambda$ and put 
$$
Q_1(z) = \prod_k\bigg( 1-\frac{z}{\lambda_{n_{2k}} + 2^{-k}} \bigg).
$$
In this case it is clear that $G_{\tilde \Lambda}$ has a unique zero near 
each of the points $\lambda_{n_k}$ which is different from $\lambda_{n_k}$. 
If we define  $\tilde \Lambda$ as the zero set  $G_{\tilde \Lambda}$, then
the set $(\rea \tilde \Lambda)\,{\rm mod}\,\mathbb{Z}$ is infinite. 

The property \eqref{skr} is obvious. Therefore, for any entire function $U$ 
such that $G_\Lambda U\in \pw$ we have $G_{\tilde \Lambda}U \in \pw$. 
Indeed, by \eqref{skr} $G_{\tilde \Lambda}U \in L^2(\mathbb{R}+i)$ 
and also $G_{\tilde \Lambda}U$ is of finite exponential type.

Now we put $F_2 = \sum_n \bar b_n k_n$, where
$$
b_n = a_n \frac{G_{\Lambda}(n)}{G_{\tilde \Lambda}(n)}.
$$
Since $|G_\Lambda(n)| \asymp |G_{\tilde \Lambda}(n)|$, $n\in\mathbb{Z}$,
we have $F_2\in \pw$ and the equation \eqref{trex} can be rewritten as
$$
\int_\mathbb{R}\frac{G_{\tilde \Lambda}(x) 
E(x)}{x-w}\overline{F_2(x)}dx  = \sum_{n\in\mathbb{Z}}
\frac{G_{\tilde \Lambda}(n)E(n)b_n}{n-w} = 0.
$$
Also it is clear that if $F_2 = \hat f_2$, $f_2\in L^2(-\pi, \pi)$, 
then $f_2-f_1 \in C^\infty[-\pi, \pi]$. 
Denote by $N$ the order of the distribution $\psi \in L_0^\perp$ 
such that $\psi(\bar f_1) \ne 0$. Choosing a sufficiently small $\delta$ 
we can achieve that the norms $\|(f_2-f_1)^{(j)}\|_{L^2(-\pi, \pi)}$, 
$j=0,1, \dots, N$, are sufficiently small and so $\psi(\bar f_2) \ne 0$.

Since the set $(\rea \tilde \Lambda)\,{\rm mod}\,\mathbb{Z}$ is infinite, we may continue 
as in the proof of Proposition \ref{old} to obtain \eqref{inteq-1}. 
\end{proof}

\subsection{Sufficient conditions for synthesability}
The following two propositions will play the key role in the proof of 
Theorem \ref{mainth}.

\begin{proposition}
Let the sequence $\Lambda$ be such that $r(\Lambda) = 2\pi$ and 
the system $\{k^{\pw}_\lambda\}_{\lambda\in\Lambda}$ has infinite codimension in $\pw$. 
If the polynomials belong to the de Branges space $\mathcal{H}_{\Lambda,\pi}$ 
and are dense there, then $\Lambda$ is synthesable. 
\label{nodef}
\end{proposition}               

\begin{proof}
Assume the contrary. Then there exists a $D$-invariant subspace $L$ 
with $\sigma(D|_L)= \Lambda$ and $I = [-\pi, \pi]$ 
such that $L\ne L_0$.

Assume first that the set 
$(\rea \Lambda)\,{\rm mod}\,\mathbb{Z}$ is infinite and let 
$E$ be the function from Proposition \ref{old}.
From \eqref{inteq} we conclude that the mixed system 
$$
 \mathcal{MS}:=\{k^{\pw}_\lambda\}_{\lambda\in\Lambda}
 \cup\biggl{\{}\frac{G_\Lambda E}{z-w}\biggr{\}}_{w\in\mathcal{Z}_E}
$$
is not complete in $\pw$. Indeed, the two parts of the system 
$\mathcal{MS}$ are mutually orthogonal. The first condition in 
\eqref{inteq} means that $F$ is orthogonal the system
$\Big\{ \frac{G_\Lambda E}{z-w}\Big\}_{w\in\mathcal{Z}_E}$ while, 
by the second condition, $F\notin \ospan\{k_\lambda\}_{\lambda\in\Lambda}$.

Let $H\in\pw$ be a nonzero function such that $H\perp\mathcal{MS}$. 
Then $H = G_\Lambda H_1$ for some function $H_1 \in \mathcal{H}_{\Lambda,\pi}$ and
using \eqref{inteq2} we get
$$
 \int_\mathbb{R}|G_\Lambda(x)|^2
 \frac{E(x)}{x-w}\overline{H_1(x)}dx=0,\qquad w\in\mathcal{Z}_E.
$$
Put 
$$
 \widetilde{\mathcal{H}}=\ospan_{\mathcal{H}_{\Lambda,\pi}}\biggl{\{}\frac{E(z)}{z-w}\biggr{\}}.
$$
Then, by Theorem \ref{spaceG}, $\widetilde{\mathcal{H}}$ is a de Branges subspace 
of the space $\mathcal{H}_{\Lambda,\pi}$. 
Since $E$ has infinitely many zeros, we have ${\rm dim}\, \mathcal{H} = \infty$.
By the hypothesis, $\Chain(\mathcal{H}_{\Lambda,\pi}) 
= \{\mathcal{P}_n:\ n\in\mathbb{N}_0\} \cup \{\mathcal{H}_{\Lambda,\pi}\}$.
Therefore, $\mathcal{H} = \mathcal{H}_{\Lambda,\pi}$. 
Since $H_1\perp \widetilde{\mathcal{H}}$ in $\mathcal{H}_{\Lambda,\pi}$,  
we have $H_1=0$, a contradiction.

Now consider the case when 
the set $(\rea \Lambda)\,{\rm mod}\,\mathbb{Z}$ is finite. In this case, by 
Proposition \ref{tech}, there exists a perturbed sequence $\tilde \Lambda$
such that \eqref{inteq-1} is satisfied. 
Note that $F\in \pw$ if and only if $F(z+i) \in \pw$ and 
$\|F\|_{L^2(\mathbb{R})} \asymp \|F\|_{L^2(\mathbb{R}+i)}$, $F\in \pw$.  
Therefore, by \eqref{skr}, for an entire function $U$
we have $G_\Lambda U \in \pw$ if and only if $G_\Lambda U \in \pw$, and so 
$\mathcal{H}_{\Lambda,\pi} = \mathcal{H}_{\tilde \Lambda,\pi}$
as sets with equivalence of norms. It follows that
the polynomials are dense in $\mathcal{H}_{\tilde \Lambda,\pi}$.
As above, this leads to a contradiction with \eqref{inteq-1}.
\end{proof}

\begin{proposition}
Let the sequence $\Lambda$ be such that $r(\Lambda) = 2\pi$ 
and the system $\{k^{\pw}_\lambda\}_{\lambda\in\Lambda}$ 
has infinite codimension in $\pw$. If de Branges space $\mathcal{H}_{\Lambda,\pi}$ 
has a de Branges subspace of codimension $1$ which contains polynomials 
as a dense subset, then $\Lambda$ is synthesable. 
\label{def1}
\end{proposition} 

\begin{proof}
We consider the case when the set 
$(\rea \Lambda)\,{\rm mod}\,\mathbb{Z}$ is infinite. In the case when   
$(\rea \Lambda)\,{\rm mod}\,\mathbb{Z}$ is finite, one should use 
Proposition \ref{tech} (in place of Proposition \ref{old}) together
with the fact that $\mathcal{H}_{\Lambda,\pi} = \mathcal{H}_{\tilde \Lambda,\pi}$
with equivalence of norms and so polynomials are dense in 
$\mathcal{H}_{\tilde \Lambda,\pi}$ up to codimension 1 as well.
\medskip

{\bf Step 1.} Assume that $L$ is not synthesable.
Let $E$ be the entire function from Proposition \ref{old} and 
consider the linear space 
$$
\mathcal{H}_{alg} = {\rm Span}\,
\bigg( \bigg\{ \frac{E(z)}{z-w}:\ w\in\mathcal{Z}_E \bigg\} \cup
\bigg\{ z^n E(z): \ z^n G_\Lambda E \in \pw \bigg\}\bigg).
$$
Then $G_\Lambda \mathcal{H}_{alg} \subset \widehat{L^\perp}$. 
By Proposition \ref{old} and formulas \eqref{inteq2},
there also exist a function $F = \hat f$, $f\in C^\infty[-\pi,\pi]$, and an entire 
function $T$ such that $G_\Lambda T\in\pw$ and
\begin{equation}
\label{bab2}
\begin{cases}
\int_\mathbb{R}G_\Lambda(x)\cdot V(x)\cdot\overline{F(x)}dx=0,
\qquad V\in\mathcal{H}_{alg}, \\
\int_\mathbb{R}G_\Lambda(x)T(x)\overline{F(x)} dx \neq 0.
\end{cases}
\end{equation} 
Moreover, we can assume that $|F(x)|=o((1+|x|)^{-10})$, $x\in\mathbb{R}$, $|x|\to \infty$. 
Indeed, otherwise in place of $f$ we can choose 
a function of the form 
$g(t) = f(t)-\sum_{j\in J} c_j e^{i \lambda_j t}$ for some finite set $J$
such that $g^{\ell}(\pm \pi) =0$, $0\le \ell \le 10$.  
\medskip

{\bf Step 2.}  Note that the chain $\Chain(\mathcal{H}_{\Lambda,\pi})$ contains 
only two infinitely dimensional subspaces: 
$\mathcal{H}_{\Lambda,\pi}$ itself and its subspace of codimension 1. Put
$$
\mathcal{H}=\ospan_{\mathcal{H}_{\Lambda,\pi}}\mathcal{H}_{alg}.
$$
Using the axiomatic description of de Branges spaces it is easy to show 
(as in Theorem \ref{spaceG}) that $\mathcal{H}$ is a de Branges space. 
Since $E$ has infinitely many zeros, ${\rm dim}\, \mathcal{H} = \infty$.
Let $H$ be the projection of $F$ to 
$\big(\ospan\{k_\lambda: \lambda\in \Lambda \}\big)^\perp$. 
Then, by \eqref{bab2}, $H\perp G_\Lambda \mathcal{H}_{alg}$ and $H\ne 0$. We 
can write $H=G_\Lambda H_1$ and so $H_1\in\mathcal{H}_{\Lambda,\pi}$.   
Then $H_1\perp\mathcal{H}$ and we conclude that 
$\mathcal{H}$ has codimension $1$ in $\mathcal{H}_{\Lambda,\pi}$.
\medskip

{\bf Step 3.} Let $A$ be a de Branges $A$-function which 
corresponds to the space $\mathcal{H}$. Since $\mathcal{H}$ 
has codimension 1 in $\mathcal{H}_{\Lambda,\pi}$, the 
function $A$ can be chosen so that $A\in \mathcal{H}_{\Lambda,\pi}$ 
by Remark \ref{afun}. 
By Theorem \ref{twoint}, $\mathcal{H}$ has no de Branges
subspace of codimension 1, and so 
$\mathcal{Z}_A$ is a uniqueness for $\mathcal{H}$.
The functions $\frac{A(z)}{z-w}$, $w\in\mathcal{Z}_A$, 
are in $\mathcal{H}$ and, by \eqref{bab2}, 
\begin{equation}
\begin{cases}
\int_\mathbb{R}\frac{G_\Lambda(x)A(x)}{x-w} \overline{F(x)}dx=0,\qquad w\in\mathcal{Z}_A,\\
\int_\mathbb{R}G_\Lambda(x)T(x)\overline{F(x)} dx \neq 0,
\label{inteq5}
\end{cases}
\end{equation}
where $G_\Lambda T\in\pw$. We may assume that 
$T$ is not a polynomial (otherwise we can replace $T$ by $T+\frac{A(z)}{z-w_0}$).
\medskip

{\bf Step 4.}
Now we expand $G_\Lambda T$ and $F$ with respect 
to the orthonormal basis $\{k_n\}_{n\in\mathbb{Z}}$
of cardinal sine functions, $k_n(z) = \frac{\sin \pi(z-n)}{\pi(z-n)}$:
$$
G_\Lambda T=\sum_{n\in\mathbb{Z}}b_nk_n,\qquad F=
\sum_{n\in\mathbb{Z}}\overline{a_n}k_n,\qquad \{a_n\},\{b_n\}\in\ell^2.
$$
Moreover, since 
$|F(x)|=o((1+|x|)^{-10})$, $x\in\mathbb{R}$, $|x|\to \infty$,
we have $\{a_n\} \in \ell^1$ and 
$\{G_\Lambda(n)A(n)a_n\} \in\ell^1$. 

We can assume that $\Lambda\cap \mathbb{Z} =\emptyset$, $a_n\ne 0$ and $b_n \ne 0$
for any $n\in\mathbb{Z}$. Otherwise, we can use a shifted basis 
$\{k_{n+\gamma}\}_{n\in\mathbb{Z}}$ with any $\gamma\in [0,1)$ and find a $\gamma$
such that the set $\mathbb{Z} + \gamma$ will have these properties. 
Moreover, we can find $\gamma$ such that, for some $c>0$,  
\begin{equation}
\dist(w,\mathbb{Z}+\gamma) \ge \frac{c}{|w|^2}, \qquad w\in\mathcal{Z}_T, \ w\ne 0.
\label{bbb}
\end{equation}
Indeed, if we take $c$ so small that 
$$
2c \sum_{w\in\mathcal{Z}_T, \, w\ne 0} |w|^{-2} <1/2, 
$$
then the intervals $[\rea w - c|w|^{-2}, \rea w + c|w|^{-2}]$, $w\ne 0$, 
considered ${\rm mod}\,\, \mathbb{Z}$ will not cover the interval $[0,1]$. 
We will assume without loss of generality that $\gamma=0$.

The first equation in \eqref{inteq5} becomes 
$$
\sum_{n\in\mathbb{Z}}\frac{G_\Lambda(n)A(n)a_n}{w-n}=0,\qquad w\in\mathcal{Z}_A.
$$
So, for some entire function $S_1$, we have
\begin{equation}
\frac{1}{\pi} \sum_{n\in\mathbb{Z}}\frac{G_\Lambda(n)A(n)a_n}{z-n}=\frac{A(z)S_1(z)}{\sin \pi z}.
\label{S1eq}
\end{equation}
On the other hand,
\begin{equation}
\frac{1}{\pi} \sum_{n\in\mathbb{Z}}\frac{b_n(-1)^n}{z-n}=\frac{G_\Lambda(z)T(z)}{\sin \pi z}.
\end{equation}
Comparing the residues at integer points we get
$$
S_1(n)T(n)=(-1)^na_nb_n.
$$
Let $W$ be a function from $\pw$ such that 
$W(n)=(-1)^na_nb_n$, $n\in \mathbb{Z}$. 
Then there exists an entire function $U$ such that
$$
S_1(z)T(z)-W(z)=U(z)\sin \pi z.
$$
Since, by \eqref{S1eq}, $AS_1\in\pw$ and $G_\Lambda A \in \pw$, 
we have $AS_1G_\Lambda T - G_\Lambda A W \in\mathcal{P}W_{2\pi}$.
Hence, $G_\Lambda A U\sin \pi z \in\mathcal{P}W_{2\pi}$ and, finally 
$G_\Lambda A U\in\pw$ and $AU \in \mathcal{H}_{\Lambda,\pi}$. 
\medskip

{\bf Step 5.} Let us show that the function $U$ is constant. 
Since $A$ is of maximal growth in $\mathcal{H}$, it is easy to see that $U$
is of zero exponential type. 
Assume that $U$ has at least one zero $u_0$. 
Since $AU \in \mathcal{H}_{\Lambda,\pi}$, the function
$\frac{AU}{z-u_0}$ belongs to the domain of multiplication by $z$
in $\mathcal{H}_{\Lambda,\pi}$ and so $\frac{AU}{z-u_0} \in \mathcal{H}$
and vanishes on $\mathcal{Z}_A$,
a contradiction to the fact that 
$\mathcal{Z}_A$ is a uniqueness set for $\mathcal{H}$.
\medskip

{\bf Step 6.} We have seen that $U=c$ is a constant function and so
$$   
\frac{S_1(z)T(z)}{\sin \pi z }=U(z)+\frac{W(z)}{\sin \pi z}=
c+\sum_{n\in\mathbb{Z}}\frac{a_nb_n}{z-n}.
$$
In addition we know that $\kappa:=(G_\Lambda T, F)_{\pw}=\sum_na_nb_n\neq0$, 
and $a_n=(-1)^n\overline{F(n)}=o(n^{-10})$.
Then it is easy to see that for $|z-l| <1/2$ we have
$$
c+\sum_{n\in\mathbb{Z}}\frac{a_nb_n}{z-n} = c+\frac{\kappa}{l} + 
\frac{a_l b_l}{z-l}+ o\Big(\frac{1}{|l|}\Big).
$$
Therefore, for sufficiently big $l \in\mathbb{Z}$. 
there exists a unique $z_l$ such that 
$|z_l-l| <1/2$,
$$
c+\sum_{n}\frac{a_nb_n}{z_l-n}=0
$$
and 
\begin{equation}
z_l=l-\frac{l a_lb_l}{cl+\kappa +o(1)} = l+o(|l|^{-9}).
\label{bbb1}
\end{equation}
By \eqref{bbb}, such $z_l$ can not be a zero of $T$ when $l$ 
is sufficiently large. Hence, $S_1$ has zeros $z_l$ 
of the form \eqref{bbb1} for all sufficiently large $l$.
Since $S_1T \in \pw +z\pw$, we conclude, by simple estimates of canonical products,    
that $T$ has at most finite number of zeros and is 
of zero type. Thus, $T$ is a polynomial, a contradiction.
\end{proof}


\subsection{End of the proof of sufficiency part in Theorem \ref{mainth}}
Now we are ready to prove the sufficiency of conditions (i) and (ii) 
of Theorem \ref{mainth}. Let $G_\Lambda = G$. 
By Lemma \ref{link}, we have $\mathfrak{E}_0(|G_\Lambda|^2)
=\mathcal{H}_{\Lambda,\pi}$. Hence, 
$\mathcal{P}\subset\mathcal{H}_{\Lambda,\pi}$ 
and $\mathcal{P}$ is dense in $\mathcal{H}_{\Lambda,\pi}$ 
or has codimension one there. In either of the cases 
$\Lambda$ is synthesable by
Proposition \ref{nodef} or by Proposition \ref{def1}, respectively.
\qed
\bigskip


\section{Proof of Theorem \ref{mainth}: Necessity}
\label{necc}

By the hypothesis, 
the system $\{e^{i\lambda t}\}_{\lambda\in\Lambda}$ has infinite defect in  
$L^2(-\pi, \pi)$. Let $G_\Lambda$ be some canonical product
with zero set $\Lambda$ such that $G_\Lambda/ G^*_\Lambda$ is a ratio 
of two Blaschke products. Then the corresponding space 
$\mathcal{H}_{\Lambda,\pi}$ is infinite-dimensional.
Moreover, multiplying, if necessary by $e^{bz}$, 
$b\in \mathbb{R}$, we can achieve that the indicator function of $G_\Lambda$
satisfies $h_{G_\Lambda}(0) = h_{G_\Lambda}(\pi)$.  
Assume that {\it it is not true that 
the polynomials are dense in $\mathcal{H}_{\Lambda,\pi}$
up to codimension 1} (it is possible, in particular, that the set 
of all polynomials is not contained in $\mathcal{H}_{\Lambda,\pi}$). 
We will show that in this case $\Lambda$ is not synthesable. 

\begin{lemma}
\label{subs}
Assume that it is not true that the polynomials are dense 
in $\mathcal{H}_{\Lambda,\pi}$ up to codimension 1. 
Then there exists a de Branges subspace
in the chain $\Chain(\mathcal{H}_{\Lambda,\pi})$ which is both infinite-dimensional 
and of infinite codimension in $\mathcal{H}_{\Lambda,\pi}$. 
\end{lemma}

\begin{proof}
Note that any finite-dimensional subspace of an arbitrary de Branges space 
is of the form $\mathcal{P}_n e^{S}$ for some function $S$ which is real on $\mathbb{R}$.
In the case of $\mathcal{H}_{\Lambda,\pi}$ 
it follows that $e^S G_\Lambda \in \pw$. By the choice of $G_\Lambda$, it follows
that $S$ is a constant. Thus, if $\mathcal{H}$ is a finite-dimensional subspace 
of $\mathcal{H}_{\Lambda,\pi}$, it is of the form $\mathcal{P}_n$. 

Now let $\mathcal{H}$ be some infinite-dimensional subspace
from $\Chain(\mathcal{H}_{\Lambda,\pi})$. If it has an infinite codimension, we are done.
If it has finite codimension, then by Theorems \ref{finth} and \ref{twoint},
its codimension can be only one and also any proper de Branges subspace of $\mathcal{H}$ 
has infinite codimension in $\mathcal{H}$. Thus, any infinite-dimensional 
subspace of  $\mathcal{H}$ would do the job.  Finally, note that
if all de Branges subspaces of $\mathcal{H}$ are finite-dimensional, then
$\mathcal{H} = \Clos_{\mathcal{H}_{\Lambda,\pi}} \mathcal{P}$. Since 
$\dim (\mathcal{H}_{\Lambda,\pi} \ominus \mathcal{H}) = 1$, we arrive to a contradiction. 
\end{proof}

From now on we assume that 
$\tilde{\mathcal{H}} \in  \Chain(\mathcal{H}_{\Lambda,\pi})$ 
is both infinite-dimensional and of infinite codimension. 
Put
$$
\tilde{\mathcal{H}}_0=\{f\in \tilde{\mathcal{H}}: f\mathcal{P} \subset 
\tilde{\mathcal{H}} \}.
$$
Note that $\tilde{\mathcal{H}}_0 \neq\emptyset$ by Theorem \ref{notleft}. 

Now we are looking for a function $T\in \mathcal{H}_{\Lambda,\pi}$ 
such that $T\notin \tilde{\mathcal{H}}$ 
and $\mathcal{P}T\subset \mathcal{H}_{\Lambda,\pi}$. 
Let $A$ and $\tilde{A}$ be the $A$-functions of  
$\mathcal{H}_{\Lambda,\pi}$ and $\tilde{\mathcal{H}}$.
By Lemma \ref{af}, $|\tilde{A}(iy)\slash A(iy)|$ 
tends to zero faster that any polynomial as $|y|\rightarrow\infty$. 
Then there exists a lacunary canonical product $S$  with 
$\mathcal{Z}_S\subset\mathcal{Z}_A$ such that  
$T:=A\slash S\in \mathcal{H}_{\Lambda,\pi}$ 
and $|T|\gtrsim |\tilde{A}|$ on $i\mathbb{R}$ whence $T\not\in \tilde{\mathcal{H}}$.
So,
$$
 G_\Lambda T\notin G_\Lambda\tilde{\mathcal{H}}.
$$
Hence, there exists $F\in\pw$ such that 
\begin{equation}
\label{ddd}
\begin{cases}
 \int_\mathbb{R}G_\Lambda(x)V(x)\overline{F(x)}dx=0 \quad 
 \text{ for any } V\in \tilde{\mathcal{H}}, \\
 \int_\mathbb{R}G_\Lambda(x)T(x)\overline{F(x)} dx \neq 0.
\end{cases} 
\end{equation}
\smallskip


\subsection{Key lemma.}
For the proof we need to find a function $F$ satisfying equations 
of the form \eqref{ddd} such that $F=\hat f$ 
for some $f \in C^\infty[-\pi,\pi]$. 

In what follows we will use the following simple observation. 

\begin{remark}
{\rm Assume that $zf \in \pw$. Then
\begin{equation}
\label{pw1}
\int_\mathbb{R} f(x) \sin \pi x\, dx =  \sum_{n\in\mathbb{Z}} (-1)^n f(n) = 0.
\end{equation}
Indeed, $\int_\mathbb{R} f(x) \sin \pi x\, dx=(zf, \pi k_0) = 0$. 
The second equality follows from the fact that $f\in \pw$ and $\int_\mathbb{R} |f| 
<\infty$, whence $g(z) = e^{i\pi z} f$ is in the Hardy space $H^1$
in the upper half-plane. Then $|g(iy)| = o(y^{-1})$, $y\to\infty$. 
On the other hand, $f(z) = \frac{\sin\pi z}{\pi} \sum_n \frac{(-1)^n f(n)}{z-n}$,
and so $|g(iy)|=e^{-\pi y}|f(iy)| \to \big|\sum_n (-1)^n f(n)\big|/(2\pi)$
as $y\to\infty$.  }
\end{remark}

\begin{lemma}
\label{smooth}
There exist an entire function $T_0$ and a 
function $F_0 \in \pw$ such that 
$\mathcal{P} T_0 \subset \mathcal{H}_{\Lambda,\pi}$,
\begin{equation}
\label{ddd0}
|F_0(n)| = o(|n|^{-N}), \qquad |n|\to \infty,
\end{equation}
for any $N>0$ and 
\begin{equation}
\label{ddd1}
\begin{cases}
 \int_\mathbb{R}G_\Lambda(x)V(x)\overline{F_0(x)}dx=0 \quad 
 \text{ for any } V\in \tilde{\mathcal{H}}, \\
 \int_\mathbb{R}G_\Lambda(x) T_0(x)\overline{F_0(x)} dx \neq 0.
\end{cases} 
\end{equation}
\end{lemma}

\begin{proof} Without loss of generality we assume that 
$(\Lambda \cup \mathcal{Z}_{\tilde A})\cap \mathbb{Z} =\emptyset$
(see the discussion in Subsection \ref{class}).
\medskip

{\bf Step 1.} 
Let $F\in \pw$ be a function satisfying \eqref{ddd}.
Recall that the functions 
$\Big\{\frac{G_\Lambda \tilde A}{z-\mu}\Big\}_{\mu \in \mathcal{Z}_{\tilde A}}$
form an orthogonal basis in $\tilde{\mathcal{H}}$.
Hence, $\int_{\mathbb{R}} G_\Lambda(x) \frac{\tilde A(x)}{x-\mu} 
\overline{F(x)}dx = 0$ for any $\mu \in \mathcal{Z}_{\tilde A}$. Let 
$F = \sum_{n\in \mathbb{Z}} \bar a_n k_n$, where $\bar a_n = F(n) \in \ell^2$, 
be the expansion of $F$ with respect to the orthonormal basis 
of cardinal sine functions. Then
$$
\sum_{n\in \mathbb{Z}} \frac{G_\Lambda(n) \tilde A(n) a_n}{\mu-n} =0, \qquad
\mu \in \mathcal{Z}_{\tilde A}.
$$
Therefore, there exists an entire function $S$ such that 
\begin{equation}
\label{rep3}
\frac{\sin \pi z}{\pi}\sum_{n\in \mathbb{Z}} \frac{G_\Lambda(n) \tilde A(n) a_n}
{z-n} = \tilde A(z) S(z).
\end{equation}
Comparing the values at $n\in \mathbb{Z}$, we get $S(n) = (-1)^n G_\Lambda(n) a_n$
and so
\begin{equation}
\label{tru}
\sum_{n\in \mathbb{Z}} \bigg|\frac{S(n)}{G_\Lambda(n)}\bigg|^2 <\infty.
\end{equation}
The idea is to divide $S$ by some lacunary product $U$ and to show that the function
$\tilde F = \sum_{n\in\mathbb{Z}} \bar d_n k_n$ where $d_n = a_n/U(n)$ 
has the required properties. This can be easily done if there exist a sequence
$\{s_k\}_{k\in\mathbb{N}} \subset \mathcal{Z}_S$ and $N>0$ such that 
$$
{\rm dist}\, (s_k,\mathbb{Z}) \gtrsim |s_k|^{-N},
$$
since in this case $|U(n)|$ grows faster than any polynomial. 
If such $s_k$ exist, one can go directly to the Step 5. Otherwise, we  
first need to modify the functions $F$ and $S$.
\medskip

{\bf Step 2.} Choose an increasing sequence $\{n_k\}_{k\in\mathbb{N}} \subset
\mathbb{N}$ such that 
${\rm dist}\, (\{n_k\}, \mathcal{Z}_S)\ge 1/10$, $k\in \mathbb{N}$. 
Such sequence exists since, otherwise, 
there will be a point of $\mathcal{Z}_S$ in $D(n, 1/10)$ for all $n \in
\mathbb{Z}$ except 
a finite number. On the other hand, it follows from \eqref{rep3} that $\tilde A S 
\in \pw+z\pw$. Thus, $\tilde A$ is at most 
a polynomial, a contradiction to the choice of $\tilde{\mathcal{H}}$. 

Shifting $\Lambda$ if necessary, we can assume without loss of generality 
that ${\rm dist}\, (n,\Lambda)\gtrsim 
(n^{2} +1)^{-1}$, $n\in\mathbb{Z}$ 
(see the proof of \eqref{bbb}). 
Since $S$ and $G_\Lambda$ are entire functions of order at most 1 and 
${\rm dist}\, (n_k, \mathcal{Z}_S \cup \Lambda)\gtrsim n_k^{-2}$, 
there exists $N>2$ such that
$$
|S(z)| \asymp |S(n_k)|, \qquad  
|G_\Lambda(z)| \asymp |G_\Lambda(n_k)|, \qquad z\in D_k:= \overline{D}\Big(n_k + \frac{1}{n_k^N}, 
\frac{1}{10n_k^{N}} \Big),
$$
with the constants independent on $k$.
This follows from standard estimates of canonical products. 
Therefore, by \eqref{tru}, $|S(z)|=o(|G(z)|)$, $z\in D_k$, $k \to \infty$.
From now on we assume $N$ to be fixed.
\medskip

{\bf Step 3.} Let $P$ be a polynomial such that $\mathcal{Z}_P \subset \mathcal{Z}_S
\setminus \mathbb{Z}$. Consider the function $F_1 = \sum_{n\in \mathbb{Z}}
\bar b_n k_n$, where $b_n = a_n/P(n)$. Let us show that
$F_1 \perp \Big\{\frac{G_\Lambda \tilde A}{z-\mu}\Big\}_{\mu \in
\mathcal{Z}_{\tilde A}}$ and so $F_1\perp G_\Lambda \tilde{\mathcal{H}}$.
As in Step 1, this orthogonality is equivalent  to the interpolation formula 
$$
\frac{\sin \pi z}{\pi}\sum_{n\in \mathbb{Z}} \frac{G_\Lambda(n) \tilde A(n) b_n}
{z-n} = 
\frac{\sin \pi z}{\pi}\sum_{n\in \mathbb{Z}} \frac{(-1)^n \tilde A(n) S(n)}
{P(n)(z-n)} = \frac{\tilde A(z) S(z)}{P(z)},
$$
which is obviously true. 

We have $F_1(n) = F(n)/\overline{P(n)} = F(n)/P^*(n)$, $n\in\mathbb{Z}$,
where $P^*(z) = \overline{P(\bar z)}$. Hence, there exists an entire function $W$
such that $P^* F_1 - F = W\sin\pi z$. Since $F, F_1\in \pw$ we conclude that
$W$ is a polynomial. 

Since $zf\in\pw$, $\int_\mathbb{R} f(x) \sin \pi x\, dx=0$ by \eqref{pw1}. 
Thus,
$$
\int_\mathbb{R} G_\Lambda T\overline{F} = 
\int_\mathbb{R} G_\Lambda T\overline{P^* F_1} - 
\int_\mathbb{R} G_\Lambda T W^* \sin \pi x  = 
\int_\mathbb{R} G_\Lambda T P \overline{F_1},   
$$
since $zG_\Lambda T W^*\in \pw$. We conclude that 
$F_1 \perp G_\Lambda V$, $V\in \tilde{\mathcal{H}}$, but,
for $T_1 = TP$ we have
$\int_\mathbb{R} G_\Lambda T_1\overline{F_1} \ne 0$. Note also that
we still have $G_\Lambda T_1 \mathcal{P} \subset \pw$. 
\medskip

{\bf Step 4.} According to Step 3, for any polynomial $P$
we can construct a new function $F_1 =
\sum_{n\in \mathbb{Z}}
\bar b_n k_n$ such that for the function $S_1 = S/P$ we have
$$
\frac{\sin \pi z}{\pi}\sum_{n\in \mathbb{Z}} \frac{G_\Lambda(n) \tilde A(n) b_n}
{z-n} = \tilde A(z) S_1(z).
$$
Also, if $K = {\rm deg}\, P$, then 
\begin{equation}
\label{cvb}
|S_1(z)| =o(|z|^{-K} |G_\Lambda(z)|), \qquad |z|\to \infty, \ z\in D_k,
\end{equation}
uniformly with respect to $z\in D_k$.

Now let us choose two lacunary canonical products
$Q_1$ and $Q_2$ with the following properties:
\medskip

(a) $\mathcal{Z}_{Q_1} \subset \Big\{n_k + \frac{1}{n_k^N}\Big\}$, \ 
$\mathcal{Z}_{Q_2} \subset \Lambda$;
\smallskip

(b) $\bigg|\dfrac{Q_1(n)}{Q_2(n)}\bigg| \lesssim \dfrac{1}{n^2+1}$, \ \  $n\in \mathbb{Z}$, 
\ \ and \ \ $\bigg|\dfrac{Q_1(z)}{Q_2(z)}\bigg| \lesssim \dfrac{1}{|z|^2}$, 
\ \ ${\rm dist}\, (z, \mathcal{Z}_{Q_2}) \gtrsim 1$;
\smallskip

(c) $ \bigg|\dfrac{Q_1(z)}{Q_2(z)}\bigg| > \dfrac{1}{|z|^K}$
\ \  when \ \ 
$\Big|z- \Big(n_k + \dfrac{1}{n_k^N}\Big) \Big| = \dfrac{1}{10n_k^N}$, \ \ 
$k\in\mathbb{N}$.
\medskip
\\ 
It is clear that, for a fixed $N$ and sufficiently large $K$, this can be achieved.

Now consider the function $S_2 = S_1 +\frac{Q_1}{Q_2} G_\Lambda$. By 
\eqref{cvb}, property (c) and the Rouch\'e theorem, 
$S_2$ has exactly one zero in the disk $D_k$ 
when $n_k + \frac{1}{n_k^N} \in \mathcal{Z}_{Q_1}$ and 
$k$ is sufficiently large. Put 
$F_2 = \sum_{n\in \mathbb{Z}} \bar c_n k_n$, where
$$
c_n = b_n +(-1)^n \frac{Q_1(n)}{Q_2(n)}.
$$
We show that $F_2 \perp G_\Lambda \tilde{\mathcal{H}}$, but
$\int_\mathbb{R} G_\Lambda T_1\overline{F_2} \ne 0$. For the first 
property  we need to show that
$F_2 \perp \Big\{\frac{G_\Lambda \tilde A}{z-\mu}\Big\}_{\mu \in
\mathcal{Z}_{\tilde A}}$.
Since $F_1 \perp G_\Lambda \tilde{\mathcal{H}}$, this is equivalent to  
$$
\sum_{n\in\mathbb{Z}}  \frac{(-1)^{n} G_\Lambda(n) 
\tilde A(n) Q_1(n)}{(\mu-n) Q_2(n)} = 0, \qquad 
\mu\in \mathcal{Z}_{\tilde A}.
$$ 
This equation would follow from the interpolation formula
\begin{equation}
\label{tru1}
\frac{1}{\pi}\sum_{n\in \mathbb{Z}} \frac{(-1)^n G_\Lambda(n) \tilde A(n) Q_1(n)}
{Q_2(n)(z-n)} = \frac{G_\Lambda(z) \tilde A(z) Q_1(z)}{Q_2(z) \sin \pi z}.
\end{equation}
Clearly, the residues at $n$ in the left-hand side and in the right-hand
side coincide and so the difference is an entire function which is,
by (b), $o(1)$ as $|z|\to \infty$ and ${\rm dist}\,(z, \mathbb{Z}) \ge 1/10$. 
Therefore, the interpolation formula \eqref{tru1} is true. Finally,
$$
\begin{aligned}
\int G_\Lambda T_1 \overline{F_2} & = \sum_{n\in\mathbb{Z}} 
G_\Lambda(n) T_1(n) \overline{F_2(n)}  \\
& = \sum_{n\in\mathbb{Z}} 
G_\Lambda(n) T_1(n) b_n + 
\sum_{n\in\mathbb{Z}} 
\frac{(-1)^n G_\Lambda(n) T_1(n) Q_1(n)}{Q_2(n)}.
\end{aligned}
$$
Note that $z G_\Lambda T_1 Q_1/Q_2 \in \pw$. Hence, by \eqref{pw1},
$\sum_{n\in\mathbb{Z}} \frac{(-1)^n G_\Lambda(n) T_1(n) Q_1(n)}{Q_2(n)} = 0$ and so
$$
\int G_\Lambda T_1 \overline{F_2} = \sum_{n\in\mathbb{Z}}  
G_\Lambda(n) T_1(n) \overline{F_1(n)} 
= \int G_\Lambda T_1 \overline{F_1}  \ne 0.
$$ 
\medskip

{\bf Step 5.} At Step 4 we constructed a function $F_2= 
\sum_{n\in \mathbb{Z}} \bar c_n k_n$ and 
the corresponding function $S_2$ such that 
$S_2(n) = (-1)^n G_\Lambda(n) c_n$ and  
$$
\frac{\sin \pi z}{\pi} \sum_{n\in \mathbb{Z}} \frac{G_\Lambda(n) \tilde A(n) c_n}
{z-n} = \tilde A(z) S_2(z).
$$
By the construction, there is a sequence $\{s_k\}$ of zeros of $S_2$
with the property 
\begin{equation}
\label{tru2}
{\rm dist}\, (s_k, \mathbb{Z}) \gtrsim |s_k|^{-N}.
\end{equation}
Since $G_\Lambda T_1 \mathcal{P} \subset \tilde{\mathcal{H}}$, there
exists a lacunary entire function $U_0$ such that 
$G_\Lambda T_1 U_0 \mathcal{P} \subset \pw$ and $x^n = o(|U_0(x)|)$, $|x|\to \infty$,
$x\in\mathbb{R}$, for any $n>0$.
Then it is clear that we can choose another lacunary product $U$ 
such that $\mathcal{Z}_{U} \subset \{s_k\}$ and 
$G_\Lambda T_1 U \mathcal{P} \subset \pw$.

Put $F_0(z) = \sum_{n\in \mathbb{Z}} \bar d_n k_n$, where
$d_n = c_n/U(n)$. Note that, by \eqref{tru2}, $|U(n)|$ tend to infinity 
super-polynomially and so \eqref{ddd0} is satisfied. Let us show that
$F_0$ satisfies \eqref{ddd1}, that is,  
$F_0 \perp G_\Lambda \tilde{\mathcal{H}}$, but
$\int_\mathbb{R} G_\Lambda T_0 \overline{F_0} \ne 0$,
where $T_0 = T_1 U$. 

For the proof of the first property we use again the interpolation formula
argument and show that
$$
\frac{1}{\pi} \sum_{n\in \mathbb{Z}} \frac{G_\Lambda(n) \tilde A(n) d_n}
{z-n}
=\frac{1}{\pi} \sum_{n\in \mathbb{Z}} \frac{(-1)^n \tilde A(n) S_2(n)}
{U(n)(z-n)}  = \frac{\tilde A(z) S_2(z)}{U(z)\sin \pi z}.
$$                   
The first equality follows from the fact that $G_\Lambda(n)d_n = (-1)^n S_2(z)/U(n)$.
To prove the second equality we use again the fact that 
the difference of the left-hand side and the right-hand side is an entire function 
which is $o(1)$ as $|z|\to \infty$,  
${\rm dist}\,(z, \mathbb{Z}) \ge 1/10$, and thus 
is identically zero. It follows that 
$F_0 \perp \Big\{\frac{G_\Lambda \tilde A}{z-\mu}\Big\}_{\mu \in
\mathcal{Z}_{\tilde A}}$.

It remains to prove that $\int_\mathbb{R} G_\Lambda T_0 
\overline{F_0} \ne 0$.
Since $F_0(n) = F_2(n)/U^*(n)$, $n\in \mathbb{Z}$, we have
$U^* F_0 - F_2 = W\sin \pi z$ for some entire function $W$. Since
$F_2$ and $F_0$ are in $\pw$ we have the estimate
$$
|W(z)| \lesssim 1+ |U^*(z)|, \qquad 
{\rm dist}\, (z,\mathbb{Z})\gtrsim 1.
$$
Recall that $G_\Lambda T_1 U \mathcal{P} \subset \pw$
and so $G_\Lambda T_1 W^* \mathcal{P} \subset \pw$, whence
$\int G_\Lambda T_1 W^* \sin \pi x = 0$. Then
$$
\int G_\Lambda T_1 \overline{U^*} \overline{F_0} =
\int G_\Lambda T_1 \overline{F_2} + \int G_\Lambda T_1 W^* \sin \pi x  = 
\int G_\Lambda T_1 \overline{F_2} \ne 0.  
$$ 
Thus, $F_0$ satisfies  \eqref{ddd1} with $T_0 = T_1U$.
\end{proof}
\medskip


\subsection{End of the proof.}
Recall that we assume that the system $\{e^{i\lambda t}\}_{\lambda\in\Lambda}$ 
has infinite defect in $L^2(-\pi, \pi)$ and that 
it is not true that the polynomials are dense 
in $\mathcal{H}_{\Lambda,\pi}$ up to codimension 1.
To arrive to a contadiction, we need to construct a non-synthesable
$D$-invariant subspace. 

Let $\tilde{\mathcal{H}}$ and $\tilde{\mathcal{H}_0}$ be the same as above.
Put 
$$
M = \big\{ f\in L^2(-\pi, \pi):\ \hat f 
\in G_\Lambda \tilde{\mathcal{H}_0} \big\}.
$$
Each element $f\in M$ defines a continuous linear functional on 
$C^\infty(-2\pi, 2\pi)$ defined by 
$$
\phi_f(h) = \int_{-\pi}^\pi h(t) f(t) dt
= \int_{\mathbb{R}} \widehat{h|_{[-\pi, \pi]}}(x) \hat f(-x) dx, 
\qquad h \in C^\infty(-2\pi, 2\pi).  
$$
We use the fact that $\widehat{\overline f}(x) = \overline{\hat f(-x)}$.

Now let 
$$
L = M^\perp = \{ h\in C^\infty(-2\pi, 2\pi):\ \phi_f(h) = 0, \, f\in M \}.
$$
By the construction, $L$ is a closed subspace of $C^\infty(-2\pi, 2\pi)$ and
$\{ h\in C^\infty(-2\pi, 2\pi):\, h|_{[-\pi, \pi]} \equiv 0 \} \subset L$. 
Cleary, $\phi_f(e^{i\lambda t}) = \hat f(\lambda) = 0$, $f\in M$, and so 
$e^{i\lambda t} \in L$. Since the set 
of common zeros of $\{\hat f:\ f\in M\}$ 
coincides with $\Lambda$, we have $\sigma(D|_L) = i \Lambda$.

Let us show that $L$ is $D$-invariant which is a consequence of the fact that 
functions in $\tilde{\mathcal{H}_0}$ can be multiplied by polynomials.
We need to show that $\int_{-\pi}^\pi h'(t) f(t) dt =0$
whenever $h\in L$,  $f\in M$. 
Since $f$ vanishes outside $[-\pi, \pi]$, the integral depends only on the values of 
$h$ inside this interval. Thus we may assume without loss of generality that 
$\supp h \subset (-\pi-\vep, \pi+\vep)$ for some small $\vep>0$.
Therefore both $F = \hat f$ and $H = \hat h$ are rapidly decaying on $\mathbb{R}$
and we have
$$
\int_{-\pi}^\pi h'(t) f(t) dt = -i \int_\RR x H(x) F(-x) dx.
$$
We have $F \in G_\Lambda \tilde{\mathcal{H}_0}$ 
and, by definition of $\tilde{\mathcal{H}_0}$,  we have 
$xF(x) \in G_\Lambda  \tilde{\mathcal{H}_0}$. Thus, 
$-i x F(-x) = \overline{\widehat{\overline{f_1}}}(x)$ for some $f_1\in M$. Hence, 
$$
\int_{-\pi}^\pi h'(t) f(t) dt = \int_{-\pi}^\pi h(t) f_1(t) dt  = 0.
$$

Now we find a function in $L\setminus L_0$. 
Let $F_0$ be the function constructed in Lemma \ref{smooth}.
Recall that $F_0 = \sum_{n\in \mathbb{Z}} \bar d_n k_n$, where
$|d_n| = o(|n|^{-N})$, $|n| \to \infty$, for any $N>0$. 
Then $F_0 = \widehat{\overline{f_0}}$ where 
$f_0 = \sum_{n\in\mathbb{Z}} d_n e^{-int}$. By the condition on 
the coefficients $d_n$, the function $f_0$ can be continued as a $2\pi$-periodic  
function which is in $C^\infty(\mathbb{R})$. 

Since $F_0$ satisfies \eqref{ddd1}, we conclude that the function
$f_0$ is annihilated by any functional from $M$ and so $f_0\in L$. 
Indeed, for $f\in M$, we have $\hat f \in G_\Lambda \tilde{\mathcal{H}}_0$ and so 
$$
\varphi_f(f_0) = \int_{-\pi}^\pi f(t) f_0(t) dt = 
\int_\RR \hat f (x) \overline{F_0(x)} dx =0.
$$
At the same time, $G_\Lambda T_0 \in \pw$ and so there exists $g \in L^2(-\pi, \pi)$
such that $G_\Lambda T_0 = \hat g$. It is clear that the functional 
$\phi_g(h) = \int_{-\pi}^\pi h(t) g(t) dt$ annihilates 
$L_0$. However, we have $\phi_g(f_0) = \int_\mathbb{R}
G_\Lambda T_0 \overline{F_0} \ne 0$, whence $f_0 \notin L_0$. 

Thus, we have shown that if the system $\{e^{i\lambda t}\}_{\lambda\in \Lambda}$
has infinite codimension in $\pw$ and $\Lambda$ is synthesable, then 
polynomials belong to $\mathcal{H}_{\Lambda,\pi}$ and are dense there 
up to codimension 1. Hence,  $G_\Lambda \in \pw$
and, in particular, $G_\Lambda$ is in the Cartwright class. Therefore, 
$G_\Lambda$ can be represented as the principal value product
\eqref{prod} up to a constant 
(there is no additional exponential factor since $G^*_\Lambda/G_\Lambda$ 
is a ratio of two Blaschke products). 
Moreover, $G_\Lambda$ is of exponential type $\pi$
and its indicator diagram is given by  $[-\pi i, \pi i]$.
Therefore, $\mathfrak{E}_0(|G_\Lambda|^2) = \mathcal{H}_{\Lambda,\pi}$
(see the proof of Lemma \ref{link}). 
Thus, polynomials are dense in $\mathfrak{E}_0(|G_\Lambda|^2)$
up to codimension 1. This completes the proof of Theorem \ref{mainth}
\qed
\bigskip


\section{Proof of Theorem \ref{mainth2}}
\label{ding}

Without loss of generality we can assume that 
$[-\pi,\pi]\subset(a,b)$ and $I_1=I_2=[-\pi,\pi]$. Put
$$
\mathcal{H}_{j, 0} =\{f \text{ entire}\,: \ G_\Lambda f \in \pw \cap 
\widehat{L^{\perp}_j}  \},\quad j=1,2.
$$
Then $\mathcal{H}_{j, 0} \subset \mathcal{H}_{\Lambda, \pi}$. 
Put $\mathcal{H}_j=\ospan_{\mathcal{H}_{\Lambda,\pi}} \mathcal{H}_{j, 0}$. 

We would like to show that either $\mathcal{H}_1 \subset \mathcal{H}_2$
or $\mathcal{H}_2 \subset \mathcal{H}_1$. 
This would be true, if we could show that 
$\mathcal{H}_1$  and  $\mathcal{H}_2$ are de Branges
subspaces of $\mathcal{H}_{\Lambda,\pi}$. 
The possibility of division by a Blaschke factor follows from the corresponding
property  for $\widehat{L^{\perp}_j}$. 
It is not clear, however, whether 
$\mathcal{H}_j$ are closed under $*$-transform. 

To overcome this difficulty, 
we use the following variant of de Branges Ordering 
\smallskip
Theorem:
\\
{\it If $\mathcal{H}_1$ and $\mathcal{H}_1$ are two closed subspaces
of a de Branges space $\mathcal{H}$ which are invariant under division 
by Blaschke factors, then there exists $a\in\mathbb{R}$ such that either 
$e^{iaz}\mathcal{H}_1 \subset  \mathcal{H}_2$
or $e^{iaz} \mathcal{H}_2 \subset \mathcal{H}_1$. }
\smallskip

This statement follows by a simple modification of the argument from the proof
of \cite[Theorem 35]{br}. Let us show that in our case $a$ must be zero. 
Indeed, assume that $e^{iaz} \mathcal{H}_1 \subset \mathcal{H}_2$. Then
for any functional $\phi\in L_1^\perp$ such that $\hat \phi \in \pw$ 
we have 
$$
e^{iaz} \hat \phi \in e^{iaz} G_\Lambda \mathcal{H}_1 \subset
G_\Lambda \mathcal{H}_2 \subset \pw.
$$ 
However, we can choose $\varphi\in L_1^\perp$ so that 
$\conv\supp\varphi=[-\pi,\pi]$ since $I = [-\pi, \pi]$ is the {\it minimal} interval
for which $L_I \subset L_1$. Therefore $e^{iaz} \hat \phi \notin \pw$ for any $a\ne 0$.
We conclude that
\begin{equation}
\label{ord22}
\mathcal{H}_1 \subset \mathcal{H}_2 \qquad \text{or} \qquad
\mathcal{H}_2 \subset \mathcal{H}_1.
\end{equation}

Now assume that $L_1\not\subset L_2$ and $L_2\not\subset L_1$. 
Then there exist two functions $f_1, f_2 \in C^\infty(a,b)$
and functionals $\psi_j \in L_j^\perp$, $j=1,2$, such that 
$\phi(\bar f_1) = 0$ for any $\phi \in L_1^\perp$, but $\psi_2(\bar f_1) \ne 0$,
and, similarly, 
$\phi(\bar f_2) = 0$ for any $\phi \in L_2^\perp$, but $\psi_1(\bar f_2) \ne 0$.
Arguing as in the proof of Proposition \ref{old} we can find
entire functions $T_j$, $F_j$, $j=1,2$, such that
$$
G_\Lambda T_j\in \pw\cap\widehat{L^\perp_j},\qquad G_\Lambda F_j\in\pw,\qquad j=1,2,
$$
$$
\begin{cases}
\int_{\mathbb{R}}G_\Lambda f\overline{F_1}=0,\qquad f\in \mathcal{H}_{1, 0},\\
\int_{\mathbb{R}}G_\Lambda T_2\overline{F_1}\neq0
\end{cases}
$$
and
$$
\begin{cases}
\int_{\mathbb{R}}G_\Lambda f\overline{F_2}=0,\qquad f\in \mathcal{H}_{2, 0},\\
\int_{\mathbb{R}}G_\Lambda T_1\overline{F_2}\neq0.
\end{cases}
$$
Consider the projections of 
$F_j$ onto $\big(\ospan\{k_\lambda: \lambda\in \Lambda \}\big)^\perp$, 
they are of the form $G_\Lambda H_j$ with $H_j \in 
\mathcal{H}_{\Lambda,\pi}$. Thus, we get 
$$
\begin{cases}
\int_{\mathbb{R}}G_\Lambda f\overline{G_\Lambda H_1}=0,\qquad f\in \mathcal{H}_{1, 0},\\
\int_{\mathbb{R}}G_\Lambda T_2\overline{G_\Lambda H_1}\neq0
\end{cases}
$$
and
$$
\begin{cases}
\int_{\mathbb{R}}G_\Lambda f\overline{G_\Lambda H_2}=0,\qquad f\in \mathcal{H}_{2, 0},\\
\int_{\mathbb{R}}G_\Lambda T_1\overline{G_\Lambda H_2}\neq0.
\end{cases}
$$
Hence, $T_j\in \mathcal{H}_{j, 0}$, $j=1,2$, 
and $T_1\not\in\mathcal{H}_2$, $T_2\not\in\mathcal{H}_1$. This contradicts
\eqref{ord22}.
\bigskip


\section{Examples\label{EX}}

In this section we give examples of synthesable subspaces.

We say that an entire function $U$ of zero exponential type
(which is not a polynomial)  belongs to the {\it Hamburger class} 
if it is real on $\RR$, has only real and simple zeros 
$\{s_k\}$, and for any $M>0$, $|s_k|^M = o(|U'(s_k)|)$, $s_k \to \infty$. 

\begin{example} 
\label{exa1}
Let $U$ be a Hamburger class function such that 
$\mathcal{Z}_U \subset \mathbb{Z}$ and the polynomials 
belong to the space 
$L^2(\mu_U)$, where $\mu_U = \sum_{n\in \mathcal{Z}_U} |U'(n)|^{-2} \delta_n$,
and are dense there.
Put $\Lambda=\mathbb{Z}\setminus \mathcal{Z}_U$. 
Then the polynomials are dense in $\mathcal{H}_{\Lambda,\pi}$.
\end{example}

Note that there are many such functions $U$, e.g., 
$U(z)=\prod_{n\in\mathbb{N}}\bigl{(}1-z\slash [n^\alpha]\bigr{)}$, $\alpha>2$.

\begin{proof}
Let $F\in\mathcal{H}_{\Lambda,\pi}$ with $G_\Lambda(z) = \dfrac{\sin \pi z}{\pi U(z)}$. 
Then  $F G_\Lambda \in\pw$ and so                                                            
$$
F(z) G_\Lambda(z) = 
\frac{F(z)\sin\pi z}{\pi U(z)}=
\frac{\sin \pi z}{\pi}
\sum_{n\in \mathcal{Z}_U}\frac{F(n)}{U'(n)}\cdot\frac{1}{z-n}
$$
by the classical Whittaker--Shannon--Kotelnikov formula.
Hence,
$$
 \|F\|^2_{\mathcal{H}_{\Lambda,\pi}}=
 \biggl{\|}\frac{F\sin\pi z}{\pi U}\biggr{\|}^2_{\pw}=
 \sum_{n\in \mathcal{Z}_U}\biggl{|}\frac{F(n)}{U'(n)}\biggr{|}^2.
$$
Since $U$ is of Hamburger class, we have $G_\Lambda P \in \pw$ for any polynomial
$P$. Hence, $\mathcal{P} \subset \mathcal{H}_{\Lambda,\pi}$ 
and for any polynomial $P \in \mathcal{P}$,
$$
 (F, P)_{\mathcal{H}_{\Lambda,\pi}}= \sum_{n\in \mathcal{Z}_U}
 \frac{F(n)\overline{P(n)}}{|U'(n)|^2}= (F, P)_{L^2(\mu_U)}.
$$
Thus, polynomials are dense in $\mathcal{H}_{\Lambda,\pi}$.
\end{proof}

The next example shows that ``Codimension One Case'' is also possible.
This situation is much more subtle and, surprisingly, it is
related to the recent result from \cite{bbb} 
which says that the spectral synthesis for exponential systems 
in $L^2(-\pi, \pi)$ always holds up to one-dimensional defect.

\begin{example}
\label{exa2}
There exists $\Lambda$ such that polynomials belong to 
$\mathcal{H}_{\Lambda,\pi}$ and have codimension one there.
\end{example}

\begin{proof} {\bf Step 1.}
We will use the construction of a nonhereditarily complete 
system from \cite{bbb}: there exist
two disjoint sets $\Lambda_1$, $\Lambda_2$ such that the system 
$\{k_\lambda\}_{\lambda\in \Lambda_1\cup\Lambda_2}$ 
is complete and minimal in $\pw$, while the mixed system
$$
\mathcal{MS}=\{k_\lambda\}_{\lambda\in\Lambda_2}\cup\biggl{\{}\frac{G_1(z)G_2(z)}
{z-\lambda}\biggr{\}}_{\lambda\in\Lambda_1}
$$
has codimension one in $\pw$ (and the defect of such mixed systems is always at most 1). 
Here $G_1, G_2$ are two canonical products with zero sets $\Lambda_1$ and $\Lambda_2$
respectively, such that $G=G_1G_2$ is the generating function 
of the system $\{k_\lambda\}_{\lambda\in \Lambda_1\cup\Lambda_2}$.

It is shown in \cite{bbb} that 
$\mathcal{MS}$ is not complete if and only if there exist entire functions $S_1$
and $S_2$ and a sequence $\{a_n\}\in\ell^2$ such that two interpolation equations hold:
\begin{equation}
\begin{cases}
\frac{G_1(z)S_1(z)}{\sin \pi z}=\frac{1}{\pi} \sum_{n\in\mathbb{Z}}\frac{G_1(n)G_2(n)a_n}{z-n},\\
\frac{G_2(z)S_2(z)}{\sin \pi z}=\frac{1}{\pi} \sum_{n\in\mathbb{Z}}\frac{(-1)^n \bar a_n}{z-n}.
\label{nonher}
\end{cases}
\end{equation}
In this case the function $G_2S_2$ will belong to the orthogonal complement to
the system $\mathcal{MS}$.

By construction in \cite{bbb}, the functions $G_1$ and $S_2$ are lacunary  
canonical products with real zeros. Since $G_2S_2\in\pw$, 
it follows that $G_2\in\pw$ and $G_2$ decays faster than any polynomial 
on $\mathbb{R}$. So, polynomials belong to the space 
$\mathcal{H}_{\Lambda_2,\pi}$. Also  note that 
multiplying the first equation in \eqref{nonher} 
by $\dfrac{z^{k+1}}{G_1(z)}$ we get
$$
\frac{z^{k+1}S_1(z)}{\sin\pi z}=\frac{1}{\pi} 
\sum_{n\in\mathbb{Z}}\frac{G_2(n)n^{k+1}a_n}{z-n}.
$$
Letting $z=0$ we obtain $\sum_{n\in\mathbb{Z}}G_2(n)n^k a_n=
\sum_{n\in\mathbb{Z}}G_2(n)n^k  \overline{G_2(n)} \overline{S_2(n)} = 0$, 
$k\in\mathbb{N}_0$, whence $h:=G_2S_2\perp z^kG_2$ in $\pw$. 
Thus the polynomials are not dense in $\mathcal{H}_{\Lambda_2,\pi}$.
\medskip

{\bf Step 2.} Let $s_0$ be an arbitrary zero of $G_1$.
We claim that 
\begin{equation}
\label{cla}
\frac{G_1(z)}{(z-s_0)(z-s)}
\in\Clos_{\mathcal{H}_{\Lambda_2,\pi}}\mathcal{P},
\qquad s\in \Lambda_1, \ s\ne s_0.
\end{equation}  
Assume that \eqref{cla} is proved. Then  we can show that
\begin{equation}
\dim\biggl{(}\mathcal{H}_{\Lambda_2,\pi}\ominus 
\ospan\biggl{\{}\dfrac{G_1(z)}{(z-s_0)(z-s)}:\ s\in 
\Lambda_1, s\ne s_0\biggr{\}}\biggr{)}  = 1.
\label{d1GS}
\end{equation}
Indeed, if $F\perp \biggl{\{}\dfrac{G_1(z)}{(z-s_0)(z-s)}:\ s\in 
\Lambda_1, s\ne s_0\biggr{\}}$ in $\mathcal{H}_{\Lambda_2,\pi}$, then 
$$
G_2F \perp \{k_{\lambda}\}_{\lambda\in\Lambda_2} \cup 
\bigg\{\frac{G(z)}{(z-s_0)(z-s)}:\ s\in 
\Lambda_1, s\ne s_0\bigg\}
$$
in $\pw$. 
Since the codimension of $\mathcal{MS}$ is 1, it is 
clear that the latter system has codimension at most $2$ in $\pw$. 
Hence, the dimension of 
$$
\mathcal{H}_{\Lambda_2,\pi}\ominus 
\ospan\biggl{\{}\dfrac{G_1(z)}{(z-s_0)(z-s)}:\ s\in 
\Lambda_1, s\ne s_0\biggr{\}}
$$
is either $1$ or $2$. By Theorem \ref{twoint}, it can be only $1$.

From \eqref{cla} and \eqref{d1GS} we immediately get that 
$\Clos_{\mathcal{H}_{\Lambda_2,\pi}}\mathcal{P}$ has
codimension 1 in $\mathcal{H}_{\Lambda_2,\pi}$.
By Proposition \ref{def1} $\Lambda_2$ is synthesable.
\medskip

{\bf Step 3.}  It remains to prove \eqref{cla}. 
We will use the fact that 
$\tilde G_1 = \dfrac{G_1}{z-s_0}$ is a lacunary canonical product with zeros 
$\{s_k\}_{k=1}^\infty$. Put
$$
H_N(z)=\prod_{k=1}^N\biggl{(}1-\frac{z}{s_k}\biggr{)}, \qquad N\ge 1.
$$
It is sufficient to show that the polynomials $\dfrac{H_N(z)}{z-s_{l}}$, 
$N\geq l$, tend to $\dfrac{\tilde G_1(z)}{z-s_l}$ in 
$\mathcal{H}_{\Lambda_2,\pi}$ for any $l$, or, equivalently,
that $\dfrac{G_2H_N}{z-s_l}$ tend to $\dfrac{G}{(z-s_0)(z-s_l)}$ in $\pw$. 
Note that, for $N>l$ and $x\in \mathbb{R}$,
$$
\bigg|\frac{H_N(x+i)}{x+i-s_l} \bigg| = 
\bigg|\frac{\tilde G_1(x+i)}{x+i-s_l}\bigg| \cdot \prod_{k>N}
\cdot \bigg|1-\frac{x+i}{s_k} \bigg|^{-1} \le \frac{|\tilde G_1(x+i)|}{|x+i -s_l|}
\cdot\frac{|s(x)|}{|x-s(x)| +1},
$$
where $s(x)$ is the point from $\{s_k\}$ closest to $x$.

Since $\bigg|\dfrac{H_N(x+i)}{x+i-s_l} \bigg| \lesssim |\tilde G_1(x+i)|$
(with the constant depending on $l$, but uniformly with respect to $N>l$) 
and $\tilde G_1G_2\in\pw$, we conclude that
$$
\int_{|x|>A} \bigg|\frac{H_N(x+i)}{x+i-s_l} - \frac{\tilde G_1(x+i)}{x+i-s_l}
\bigg|^2 |G_2(x+i)|^2dx \lesssim 
\int_{|x|>A}  |\tilde G_1(x+i)G_2(x+i)|^2dx \to 0, \quad A\to\infty.
$$                                          
Since $\dfrac{H_N(x+i)}{x+i-s_l} - \dfrac{\tilde G_1(x+i)}{x+i-s_l}$
converges to zero uniformly on compact subsets, it follows that choosing first 
$A$ and then a sufficiently large $N$, we can make the
integral 
$$
\int_{\mathbb{R}} \bigg|\frac{H_N(x+i)}{x+i-s_l} - \frac{\tilde G_1(x+i)}{x+i-s_l}
\bigg|^2 |G_2(x+i)|^2dx
$$
as small as we wish. Since $\|F\|_{L^2(\mathbb{R})}
\asymp \|F\|_{L^2(\mathbb{R}+i)}$, $F\in \pw$, 
we conclude that  the polynomials $\dfrac{H_N(z+i)}{z+i-s_l}$
converge to  $\dfrac{G_1(z)}{(z-s_0)(z-s)}$ 
which completes the proof of \eqref{cla}.
\end{proof}


\begin{thebibliography}{99}

\bibitem{abbb} E. Abakumov, A. Baranov, Yu. Belov, {\it Localization of 
zeros for Cauchy transforms}, 
Int. Math. Res. Notices {\bf 2015} (2015), 15, 6699--6733.

\bibitem{abu}
N. F. Abuzyarova, {\it Spectral synthesis for the differentiation 
operator in the Schwartz space},
Math. Notes {\bf 102} (2017), 1--2, 137--148. 

\bibitem{abb} A. Aleman, A. Baranov, Yu. Belov, 
{\it Subspaces of $C^\infty$ invariant under the differentiation},
J. Funct. Anal. {\bf 268} (2015), 2421--2439.

\bibitem{alkor}
A. Aleman, B. Korenblum, {\it Derivation-invariant subspaces of
$C^\infty$}, Comput. Methods Funct. Theory {\bf 8} (2008), 2,
493--512.

\bibitem{bbb0}  A. Baranov, Yu. Belov, A. Borichev, 
{\it A restricted shift completeness problem},
J. Funct. Anal. {\bf 263} (2012), 1887--1893.

\bibitem{bbb}
A. Baranov, Yu. Belov, A. Borichev, {\it Hereditary completeness for systems of
exponentials and reproducing kernels}, Adv. Math. {\bf 235} (2013),  525--554.

\bibitem{bbb1}
A. Baranov, Yu. Belov, A. Borichev,
{\it Spectral synthesis in de Branges spaces},
Geom. Funct. Anal. (GAFA) {\bf 25} (2015), 2, 417--452.

\bibitem{B1}
Yu. Belov, {\it Complementability of exponential systems}, 
C. R. Math. Acad. Sci. Paris, 353 (2015), 215-218;

\bibitem{BM} A. Beurling, P. Malliavin, {\it On the closure of characters and the zeros of entire
functions}, Acta Math. \textbf{118} (1967), 79--93.

\bibitem{br} L. de Branges, {\it Hilbert Spaces of Entire
Functions}, Prentice--Hall, Englewood Cliffs, 1968.

\bibitem{hj}
V. Havin, B. J\"oricke, \textit{The Uncertainty Principle in Harmonic Analysis},
Springer-Verlag, Berlin, 1994.

\bibitem{lev} B. Ya.~Levin, {\it Lectures on Entire Functions},
Transl. Math. Monogr. Vol. 150, AMS, Providence, RI, 1996.					 

\bibitem {mp} N. Makarov, A. Poltoratski,
\emph{Meromorphic inner functions, Toeplitz kernels and
the uncertainty principle}, Perspectives in Analysis,
Math. Phys. Stud. 27, Springer, Berlin, 2005, 185--252.

\bibitem{mp1} M. Mitkovski, A. Poltoratski,
{\it P\'olya sequences, Toeplitz kernels and gap theorems},
Adv. Math. {\bf 224} (2010), 1057--1070.

\bibitem{os} J. Ortega-Cerd\`a, K. Seip, {\it Fourier frames}, 
Ann.\ Math.\ {\bf 155} (2002), 789--806.

\bibitem {red} R. M. Redheffer, {\it Completeness of sets of 
complex exponentials},  Adv. Math. {\bf  24} (1977), 
1, 1--62. 

\bibitem{rom} R. Romanov, {\it Canonical systems and de Branges spaces},
arXiv:1408.6022.

\bibitem{Sch}  L. Schwartz,  {\it Th\'eorie g\'en\'erale des 
fonctions moyenne-p\'eriodiques}, Ann. of Math.  (2)  {\bf 48} (1947), 857--929.

\bibitem{young} R. Young, {\it On complete biorthogonal system},
Proc.\ Amer.\ Math.\ Soc.\ \textbf{83} (1981), 3, 537--540.

\end{thebibliography}
\end{document}